\documentclass[12pt]{amsart}
\usepackage{amsfonts}
\usepackage[pdftex]{graphicx}
\usepackage{amsmath}
\usepackage{amssymb}
\usepackage{numinsec}
\usepackage{tipa}
\usepackage{mathrsfs}
\usepackage{esint}
\usepackage{textcomp}
\usepackage{fullpage}
\usepackage{a4}

\usepackage{fancyhdr}
\usepackage[driverfallback=pdfmx,,colorlinks,bookmarksopen,bookmarksnumbered,citecolor=blue, urlcolor=blue]{hyperref}
\makeatother
\newtheorem{theorem}{Theorem}[section]
\newtheorem{lemma}[theorem]{Lemma}
\newtheorem{proposition}[theorem]{Proposition}
\theoremstyle{definition}

\theoremstyle{remark}

\numberwithin{equation}{section}

\newcommand{\C}{\mathbb{C}}
\newcommand{\R}{\mathbb{R}}
\newcommand{\dint}{\displaystyle\int}
\input{tcilatex}

\begin{document}
\title{On the stability  of the  Bresse system with frictional damping
}

\author{Tej-Eddine Ghoul}
\address{New-York University Abu Dhabi
Department of Mathematics 
NYUAD, Saadiyat Island
PO Box 129188, Abu Dhabi, United Arab Emirates, Email: teg6@nyu.edu}
\author{Moez Khenissi }
\address{Ecole Sup\'erieure des Sciences
et de Technologie de Hammam Sousse
Rue Lamine El Abbessi
4011 Hammam Sousse, Tunisia}
\author{Belkacem Said-Houari }
\address{ALHOSN University, Mathematics and Natural Sciences Department, PO Box 38772, Abu Dhabi, United Arab Emirates,Email: bsaidhouari@gmail.com}
\maketitle

\begin{abstract}
In this paper, we consider the Bresse system with frictional damping terms and prove some optimal decay results for the $L^2$-norm of the solution and its higher order derivatives. In fact, if  we consider just one damping term acting on the second equation of the solution, we show that the solution does not decay at all. On the other hand, by considering one damping term alone acting on the third equation, we show that this damping term is strong enough to stabilize the whole system. In this case, we found a completely new stability number that depends on the parameters in the system. 
 In addition, we prove the optimality of the results by using  eigenvalues expansions. We have also improved the result obtained recently in \cite{Said_Soufyane_2014_1} for the two  damping terms case and get better decay estimates. Our obtained results have been proved under some assumptions on the wave speeds of the three equations in the Bresse system.  
\end{abstract}

\section{Introduction}
In this paper, we consider the Cauchy problem of the Bresse system with frictional damping 
\begin{equation}
\left\{
\begin{array}{l}
\varphi _{tt}-\left( \varphi _{x}-\psi -lw\right) _{x}-k^{2}l\left(
w_{x}-l\varphi \right) =0,\vspace{0.2cm} \\
\psi _{tt}-a^{2}\psi _{xx}-\left( \varphi _{x}-\psi -lw\right) +\gamma
_{1}\psi _{t}=0,\vspace{0.2cm} \\
w_{tt}-k^{2}\left( w_{x}-l\varphi \right) _{x}-l\left( \varphi _{x}-\psi
-lw\right)+\gamma_2 w_t =0,%
\end{array}%
\right.  \label{Bresse_main_system}
\end{equation}%
with the initial data
\begin{equation}
\left( \varphi ,\varphi _{t},\psi ,\psi _{t},w,w_{t}\right) \left(
x,0\right) =\left( \varphi _{0},\varphi _{1},\psi _{0},\psi
_{1},w_{0},w_{1}\right),  \label{Initial_data}
\end{equation}%
where $(x,t)\in \mathbb{R}\times \mathbb{R}^{+}$ and $a,\,l,\gamma
_{1},\,\gamma_2$ and $ k$ are positive constants.
The functions $w(x,t),\ \varphi(x,t)$ and $\psi(x,t)$  are, respectively, the longitudinal displacements, the vertical displacement of the beam and  the rotation angle of the linear filaments material. 

The decay rate of the solution of the  problem \eqref{Bresse_main_system}-\eqref{Initial_data} has been first studied by Soufyane and Said-Houari in \cite{Said_Soufyane_2014_1} and investigated the relationship between the frictional damping terms, the wave speeds of propagation and their influence on the decay rate of the solution. In addition, they showed that the $L^2$-norm of the solution decays with the following rate:
\begin{itemize}
\item For $a=1$, we have 
\begin{equation}
\left\Vert \partial _{x}^{j}U\left( t\right) \right\Vert _{L^{2}}\leq
C\left( 1+t\right) ^{-1/4-j/2 }\left\Vert
U_{0}\right\Vert _{L^{1}}+C\left( 1+t\right) ^{-\ell/2}\left\Vert \partial
_{x}^{j+\ell}U_{0}\right\Vert _{L^{2}},\label{Decay_U_estimate_L_2_1_gamma_0_1_gamma_2_0}
\end{equation}
\item For $a\neq 1$, we have 
\begin{equation}
\left\Vert \partial _{x}^{j}U\left( t\right) \right\Vert _{L^{2}}\leq
C\left( 1+t\right) ^{-1/4-j/2 }\left\Vert
U_{0}\right\Vert _{L^{1}}+C\left( 1+t\right) ^{-\ell/4}\left\Vert \partial
_{x}^{j+\ell}U_{0}\right\Vert _{L^{2}}, 
\label{Decay_U_estimate_L_2_1_0}
\end{equation}
\end{itemize}
where $C$ is a positive constant,  $U(x,t)=(\varphi_{x}-\psi-lw,\varphi_{t},a\psi_{x},\psi_{t},k(w_{x}-l\varphi),w_{t})^{T}(x,t)$ and $j$ and $\ell$ are positive integers. 
As we have seen  both  estimates contain some regularity losses. These regularity losses  will make the nonlinear problem  difficult to handle. See for instance \cite{IK08} and \cite{Racke_Said_2012_1} where  similar difficulties hold for the Timoshenko system.  

The main  open questions stated in \cite{Said_Soufyane_2014_1} were: 
\begin{itemize}
\item Is it possible to remove the regularity loss in the above estimates, especially in \eqref{Decay_U_estimate_L_2_1_gamma_0_1_gamma_2_0}?
\item Can we prove some decay estimates by considering just one damping term in the system. That is  for $\gamma_1=0$ or $\gamma_2=0$? 
\end{itemize}
The main goal of this paper is to give answers to the above questions.
 Indeed, we can summarize   our results as follows:

First, 
 for $\gamma_1>0$ and $\gamma_2>0$, we refined the decay estimates in \eqref{Decay_U_estimate_L_2_1_gamma_0_1_gamma_2_0} and \eqref{Decay_U_estimate_L_2_1_0} and instead, we obtained the following: 
\begin{itemize}
\item For $a=1$, we have 
\begin{equation}
\left\Vert \partial _{x}^{j}U\left( t\right) \right\Vert _{L^{2}}\leq
C\left( 1+t\right) ^{-1/4-j/2 }\left\Vert
U_{0}\right\Vert _{L^{1}}+Ce^{-ct}\left\Vert \partial
_{x}^{j}U_{0}\right\Vert _{L^{2}},\label{Decay_U_estimate_L_2_1_gamma_0_1_gamma_2_2}
\end{equation}
\item For $a\neq 1$, we have 
\begin{equation}
\left\Vert \partial _{x}^{j}U\left( t\right) \right\Vert _{L^{2}}\leq
C\left( 1+t\right) ^{-1/4-j/2 }\left\Vert
U_{0}\right\Vert _{L^{1}}+C\left( 1+t\right) ^{-\ell/2}\left\Vert \partial
_{x}^{j+\ell}U_{0}\right\Vert _{L^{2}}.\label{Decay_U_estimate_L_2_1_section_6_gamma_2_2}
\end{equation}
\end{itemize}
In addition, we showed  the optimality of  the above  estimates \eqref{Decay_U_estimate_L_2_1_gamma_0_1_gamma_2_2} and \eqref{Decay_U_estimate_L_2_1_section_6_gamma_2_2} by exploiting  the eigenvalues expansion. Consequently, those estimates, under the same assumptions on  the initial data are optimal and cannot be improved. The  proof is  essentially based on  refinements  of the Lyapunov functionals used in \cite{Said_Soufyane_2014_1}, which  give the following decays for the Fourier image of the solution $\hat{U}(\xi,t)$: 
\begin{equation}
\vert \hat{U}\left( \xi ,t\right) \vert ^{2}\leq \left\{
\begin{array}{ll}
Ce^{-c\rho _{1}\left( \xi \right) t}\vert \hat{U}\left( \xi ,0\right)
\vert ^{2}, & \text{ if}\quad a=1,\quad \rho _{1}\left( \xi \right) =\dfrac{\xi ^{2}}{1+\xi ^{2} },\vspace{0.1cm} \\
Ce^{-c\rho _{2}\left( \xi \right) t}\vert \hat{U}\left( \xi ,0\right)
\vert ^{2}, & \text{ if}\quad a\neq 1,\quad \rho _{2}\left( \xi \right) =\dfrac{\xi ^{2}}{ 1+\xi ^{2}+\xi^4 }. 
\end{array}%
\right.  \label{estimate_fourier_semigroup1_gamma_0_gamma_0_0}
\end{equation}
Since $\rho_1(\xi)$ is behaving like $|\xi|^2$ near  zero, then the decay rate of the solution is the same as the one of the heat kernel. 
 On the other hand,  as $|\xi|$ goes to infinity,  $\rho_1(\xi)$ does not go  to zero, which keeps  the dissipation  still effective at infinity and  prevents the regularity loss in \eqref{Decay_U_estimate_L_2_1_gamma_0_1_gamma_2_2}.
 When $a\neq 1$,  $\rho_2(\xi)$ has the same behavior as $\rho_1(\xi)$ near zero, but as $|\xi|$ tends to infinity,   $\rho_2(\xi)$ goes to zero, which induces  the regularity loss at infinity. 

Second,  for $\gamma_2=0$, we showed, despite  the presence  of the dissipation term $\gamma_1\psi_t (x, t)$ in the second equation, that the solution  $U(x,t)$ of the system does not decay at all.  This is due to the weakness of the coupling term $\left( \varphi _{x}-\psi
-lw\right)$  which is of order zero in the second equation. For instance, this can be viewed when $l=0$, where the first two equations in the system \eqref{Bresse_main_system} reduces to the Timoshenko system  and the third one becomes a conservative decoupled wave equation
\begin{equation}
\left\{
\begin{array}{l}
\varphi _{tt}-\left( \varphi _{x}-\psi \right) _{x}=0,\vspace{0.2cm} \\
\psi _{tt}-a^{2}\psi _{xx}-\left( \varphi _{x}-\psi \right) +\gamma
_{1}\psi _{t}=0,\vspace{0.2cm}\\
w_{tt}-k^2w_{xx}=0. 
\end{array}%
\right.  \label{Timosh_1}
\end{equation}
If we pick  initial  data $w_0$ and $w_1$ such that the corresponding solution  $w(x,t)=0$, for all $t>0$, then the  system \eqref{Timosh_1} will decay like the Timoshenko one. 
 In fact, we proved  (see Theorem \ref{Main_Theorem_1}) the following estimate 
\begin{equation}\label{No_decay_Estimate}
\left\Vert \partial _{x}^{j}U\left( t\right) \right\Vert _{L^{2}}\leq
C\left\Vert \partial
_{x}^{j}U_{0}\right\Vert _{L^{2}}. 
\end{equation}

Third, for $\gamma_1=0$, thanks to the strong  coupling term $k^{2}\left( w_{x}-l\varphi \right) _{x}$,  which is of order one, in the third equation,  the effect of the dissipation term  
$\gamma_2w_t(x,t)$ will be propagated to the other equations of the system. More precisely, 
 we proved that the solution decays with the following rate:
\begin{itemize}
\item For $a=k=1$, we have 
\begin{equation}
\left\Vert \partial _{x}^{j}U\left( t\right) \right\Vert _{L^{2}}\leq
C\left( 1+t\right) ^{-1/4-j/2 }\left\Vert
U_{0}\right\Vert _{L^{1}}+C(1+t)^{-\ell/2}\left\Vert \partial
_{x}^{j+\ell+1}U_{0}\right\Vert _{L^{2}}, 
 \label{Decay_U_estimate_L_2_1_gamma_0_1_gamma_2_k_0}
\end{equation}
\item For $a\neq 1$, we have 
\begin{equation}
\left\Vert \partial _{x}^{j}U\left( t\right) \right\Vert _{L^{2}}\leq
C\left( 1+t\right) ^{-1/4-j/2 }\left\Vert
U_{0}\right\Vert _{L^{1}}+C\left( 1+t\right) ^{-\ell/2}\left\Vert \partial
_{x}^{j+\ell+3}U_{0}\right\Vert _{L^{2}},
\label{Decay_U_estimate_L_2_1_section_6_gamma_2_k_0}
\end{equation}
\end{itemize} 
The above results were proved under the following additional assumptions on the coefficients of the system 
\begin{equation}\label{New_Condition}
(k^2-1)l^2\neq 1. 
\end{equation}
The above stability assumption is  completely new in this framework and it is satisfied for instance if $k=1$ or $l=0$.  

To prove the above estimates, we present a new method that is based solely on the eigenvalues expansion and it does not require the knowledge of eigenvectors and eigenspaces at all. (See Lemma \ref{Theorem_Exponential_avoid_Jordan}).  Indeed, we split the frequency space into three regions, low frequencies, middle frequencies and high frequencies, where in each region we derive  the eigenvalues expansion and estimate the Fourier image of the solution in each region. In particular, this method avoids the use of the Jordan canonical form in computing the exponential of the matrix of the solution which is a heavy  task to accomplish, especially, for large systems. 
 The  method that we introduce   here seems new and can be easily extended to other problems. 
 
 Now, before closing this section, let us recall some related results.  The initial boundary value problem associated to \eqref{Bresse_main_system} has been considered by many peoples recently.   Liu and Rao \cite{LRa09} investigated  the Bresse system with two
different dissipative mechanism, given by two temperatures coupled to the
system. The authors proved  that the exponential decay exists only when the
velocities of the wave propagations are the same. If the wave speeds are
different they showed that the energy of the system decays polynomially to
zero with the rate $t^{-1/2}$ or $t^{-1/4}$, provided that the
boundary conditions is of Dirichlet--Neumann--Neumann or
Dirichlet--Dirichlet--Dirichlet type, respectively.  This result was improved
by Fatori and Mu\~ noz Rivera \cite{FaR10}, where they showed that, in
general, the Bresse system is not exponentially stable but there exists
polynomial stability with rates that depend on the wave speed propagations and the
regularity of the initial data.

For the Cauchy problem, there are only few results. The first paper that dealt with the Cauchy problem for  the Bresse system is \cite{Said_Soufyane_2014_1},  where the authors investigated the relationship between the frictional damping terms, the wave speeds of propagation and their influences on the decay rate of the solution. In addition, they showed (among other results) the estimates \eqref{Decay_U_estimate_L_2_1_gamma_0_1_gamma_2_0} and \eqref{Decay_U_estimate_L_2_1_0}. The Bresse--Fourier system (Bresse system coupled with the Fourier law of heat conduction)  has been investigated by Said-Houari and Soufyane in  \cite{Said_Soufyan_2014_MMAS}, the Bresse--Cattaneo by Said-Houari and Hamadouche \cite{Said_Katia_2015} and the Bresse system in thermoelasticity of type III by   Said-Houari and Hamadouche \cite{Said_Katia_2015_2} where in these three systems, some decay estimates have been proved under some appropriate assumptions on the coefficients of the systems. 
 
 This paper is organized as follows: in Section \ref{Section_2}, we state the problem. Section \ref{Section_Non_decaying} is devoted to the case $\gamma_1>0$ and $\gamma_2=0$, where we prove that the solution does not decay at all.  In Section \ref{Asym_Expan_Section}, we study the case $\gamma_1>0$ and $\gamma_2>0$, where we  prove an optimal decay rate using the Lyapunov functional method. Finally in  Section \ref{Section_5} we  investigate the case where $\gamma_1=0$ and $\gamma_2>0$ and show, through a new method based on the use of the eigenvalues expansion combined with Fourier splitting method, the optimal decay rate of the solution. 
 
\section{Statement of the problem}\label{Section_2}

In this section, we state the  problem and introduce  some materials  that will be needed later. Let us first rewrite system (%
\ref{Bresse_main_system})-(\ref{Initial_data}) as a first order system of
the form
\begin{equation}
\left\{
\begin{array}{l}
U_{t}+AU_{x}+LU=0,\vspace{0.2cm} \\
U\left( x,0\right) =U_{0},%
\end{array}%
\right.  \label{First_order_system}
\end{equation}%
where $A$ is a real symmetric matrix and $L$ is a non-negative (non-symmetric) definite  
matrix. To this end, we introduce the following variables:
\begin{equation}\label{Change_Variables}
v=\varphi _{x}-\psi -lw,\quad u=\varphi _{t},\quad z=a\psi _{x},\quad y=\psi
_{t},\quad \phi =k\left( w_{x}-l\varphi \right) ,\quad \eta =w_{t}.
\end{equation}%
Consequently, system (\ref{Bresse_main_system}) can be rewritten into the
following first order system
\begin{equation}
\left\{
\begin{array}{ll}
v_{t}-u_{x}+y+l\eta =0,\vspace{0.2cm} &  \\
u_{t}-v_{x}-lk\phi =0,\vspace{0.2cm} &  \\
z_{t}-ay_{x}=0,\vspace{0.2cm} &  \\
y_{t}-az_{x}-v+\gamma _{1}y=0,\vspace{0.2cm} &  \\
\phi _{t}-k\eta _{x}+lku=0,\vspace{0.2cm} &  \\
\eta _{t}-k\phi _{x}-lv+\gamma_2\eta  =0 &
\end{array}%
\right.  \label{Transformed_system}
\end{equation}%
and the initial conditions (\ref{Initial_data}) takes the form
\begin{equation}
\left( v,u,z,y,\phi ,\eta \right) \left( x,0\right) =\left(
v_{0},u_{0},z_{0},y_{0},\phi _{0},\eta _{0}\right),
\label{initial_conditions_2}
\end{equation}%
where%
\begin{equation*}
v_{0}=\varphi _{0x}-\psi _{0}-lw_{0},\quad u_{0}=\varphi _{1},\quad
z_{0}=a\psi _{0x},\quad y_{0}=\psi _{1},\quad \phi
_{0}=kw_{0x}-lk\varphi _{0},\quad \eta =w_{1}.
\end{equation*}%
System (\ref{Transformed_system})-(\ref{initial_conditions_2}) is equivalent
to system (\ref{First_order_system}) with%
\begin{equation*}
U=\left(
\begin{array}{c}
v \\
u \\
z \\
y \\
\phi \\
\eta%
\end{array}%
\right) ,\,\ A=\left(
\begin{array}{cccccc}
0 & -1 & 0 & 0 & 0 & 0 \\
-1 & 0 & 0 & 0 & 0 & 0 \\
0 & 0 & 0 & -a & 0 & 0 \\
0 & 0 & -a & 0 & 0 & 0 \\
0 & 0 & 0 & 0 & 0 & -k \\
0 & 0 & 0 & 0 & -k & 0%
\end{array}%
\right) ,\,\ L=\left(
\begin{array}{cccccc}
0 & 0 & 0 & 1 & 0 & l \\
0 & 0 & 0 & 0 & -lk & 0 \\
0 & 0 & 0 & 0 & 0 & 0 \\
-1 & 0 & 0 & \gamma_1  & 0 & 0 \\
0 & lk & 0 & 0 & 0 & 0 \\
-l & 0 & 0 & 0 & 0 & \gamma_2%
\end{array}%
\right)
\end{equation*}%
and $U_{0}=\left( v_{0},u_{0},z_{0},y_{0},\phi _{0},\eta _{0}\right) ^{T}$.

By taking the Fourier transform of (\ref{First_order_system}) we obtain the
following Cauchy problem of a first order system%
\begin{equation}
\left\{
\begin{array}{l}
\hat{U}_{t}+i\xi A\hat{U}+L\hat{U}=0,\vspace{0.2cm} \\
\hat{U}\left( \xi ,0\right) =\hat{U}_{0}.%
\end{array}%
\right.  \label{Fourier_system}
\end{equation}
The solution of \eqref{Fourier_system} is given by 
\begin{equation*}
\hat{U}(\xi,t)=e^{\Phi(i\xi)t}\hat{U}_0(\xi),
\end{equation*}
where 
\begin{equation}\label{Phi_Matrix}
\Phi(\zeta)=-(L+\zeta A),\qquad \zeta=i\xi\in \C. 
\end{equation}

\section{Non decaying solutions for  $\gamma_2=0$ }\label{Section_Non_decaying}
In this section, we  assume that $\gamma_2=0$ and show that the damping term $\gamma_1\psi _{t}$ is not strong enough to obtain a decay rate of the solution. Our main result in this section reads as follows:
 
 \begin{theorem}[No decay rates]\label{Main_Theorem_1}
Let $s$ be a nonnegative integer. Let $U(x,t)$ be the solution of \eqref{First_order_system}. Assume that $U_0\in H^s(\R)$,  then the  following estimate holds: 
\begin{equation}\label{Estimate_1_0}
\left\Vert \partial _{x}^{j}U\left( t\right) \right\Vert _{L^{2}}\leq
C\left\Vert \partial
_{x}^{j}U_{0}\right\Vert _{L^{2}}, \quad j=0,1,...,s. 
\end{equation}
where $C$ is a  positive constant.  
\end{theorem}
\begin{proof}
To prove  \eqref{Estimate_1_0}, we use the dissipation of the energy 
\begin{equation*}
\hat{E}(\xi,t)=\frac{1}{2}|\hat{U}(\xi,t)|^2
\end{equation*}
of \eqref{Fourier_system} which satisfies    (see Lemma \ref{diss_Lemma}) 
\begin{equation}
\dfrac{d\hat{E}(\xi,t)}{dt}=-\gamma _{1}\left\vert \hat{y}\right\vert
^{2},\qquad \forall t\geq 0. 
\label{estimate derive}
\end{equation}
Hence, the solution of \eqref{Fourier_system} satisfies 
\begin{equation}\label{Pointwise_estimate_1}
|\hat{U}(\xi,t)|^2 \leq |\hat{U}(\xi,0)|^2,\qquad \forall t\geq 0. 
\end{equation}
Therfore, \eqref{Estimate_1_0} follows from the Plancherel theorem.

To justify the nondecay rate in \eqref{Estimate_1_0},  
 we  show that there exists at least one eigenvalue  $\lambda(\zeta)$ of $\Phi$ such that $Re(\lambda(\zeta))=0$. We compute the characteristic polynomial of $\Phi(\zeta)$ to get 
\begin{eqnarray}\label{Characteristic}
f(\lambda,\zeta)&=&\det(\lambda I-\Phi(\zeta))\notag\\
&=&\lambda^6+\gamma_1\lambda^5
+\Big\{(k^2+1)(l^2-\zeta ^2) +1-a^2\zeta^2\Big\}\lambda^4\notag\\
&&+\gamma_1(k^2+1)(  l^2-\zeta ^2 )\lambda^3\notag\\
&&+(l^2-\zeta^2) \Big\{k^2(l^2-\zeta^2)+\Big(k^2-a^2(k^2+1)\zeta^2\Big)\Big\}\lambda^2\notag\\
&&+ \gamma_1 k^2( l^2-\zeta ^2)^2\lambda
-a^2k^2\zeta^2(l^2-\zeta^2)^2.
\end{eqnarray}
 The above polynomial can be rewritten as 
\begin{equation}\label{Polynomial_Factor}
\left(\lambda^2+k^2 (l^2- \zeta^2)\right)
 \Big\{\lambda ^4+\gamma_1 \lambda ^3 +\left(l^2+1-\zeta^2(a^2+1)\right)\lambda ^2 +  \gamma_1\left( l^2- \zeta^2\right)\lambda
 -a^2  \zeta^2(l^2-\zeta^2)\Big\}. 
\end{equation}
Since $\zeta=i\xi$, it is clear  from \eqref{Polynomial_Factor} that for all $\xi\in\R$,  the matrix $\Phi(\zeta)$ has two pure imaginary roots. 
Consequently, according to the stability theory of linear ODE systems (see \cite[p.71]{GT})  the solution of \eqref{Fourier_system} doesn't go to zero. 
\end{proof}

\section{Optimal decay rates for $\gamma_i> 0,\, i=1,2$}\label{Asym_Expan_Section}
In this section, we consider   the system 
\begin{equation}
\left\{
\begin{array}{l}
\varphi _{tt}-\left( \varphi _{x}-\psi -lw\right) _{x}-k^{2}l\left(
w_{x}-l\varphi \right) =0,\vspace{0.2cm} \\
\psi _{tt}-a^{2}\psi _{xx}-\left( \varphi _{x}-\psi -lw\right) +\gamma
_{1}\psi _{t}=0,\vspace{0.2cm} \\
w_{tt}-k^{2}\left( w_{x}-l\varphi \right) _{x}-l\left( \varphi _{x}-\psi
-lw\right)+\gamma_2w_t =0,%
\end{array}%
\right.  \label{Bresse_main_system_2}
\end{equation}%
with the initial data
\begin{equation}
\left( \varphi ,\varphi _{t},\psi ,\psi _{t},w,w_{t}\right) \left(
x,0\right) =\left( \varphi _{0},\varphi _{1},\psi _{0},\psi
_{1},w_{0},w_{1}\right) ,  \label{Initial_data_2}
\end{equation}%
where $(x,t)\in \mathbb{R}\times \mathbb{R}^{+}$ and $a,\,l,\gamma
_{1}, \gamma_2$ and $ k $ are positive constants. The main result in this section reads as:
\begin{theorem}[Optimal decay rates]\label{Main_Theorem_2}
Let $s$ be a nonnegative integer. Let $U(x,t)$ be the solution of \eqref{First_order_system}. Assume that $U_0\in H^s(\R)\cap L^1(\R)$,  then the following decay estimates hold:
\begin{itemize}
\item For $a=1$, we have 
\begin{equation}
\left\Vert \partial _{x}^{j}U\left( t\right) \right\Vert _{L^{2}}\leq
C\left( 1+t\right) ^{-1/4-j/2 }\left\Vert
U_{0}\right\Vert _{L^{1}}+Ce^{-ct}\left\Vert \partial
_{x}^{j}U_{0}\right\Vert _{L^{2}}, \quad j=0,1,...,s. \label{Decay_U_estimate_L_2_1_gamma_0_1_gamma_2}
\end{equation}
\item For $a\neq 1$, we have 
\begin{equation}
\left\Vert \partial _{x}^{j}U\left( t\right) \right\Vert _{L^{2}}\leq
C\left( 1+t\right) ^{-1/4-j/2 }\left\Vert
U_{0}\right\Vert _{L^{1}}+C\left( 1+t\right) ^{-\ell/2}\left\Vert \partial
_{x}^{j+\ell}U_{0}\right\Vert _{L^{2}}, \quad j=0,1,...,s-\ell .\label{Decay_U_estimate_L_2_1_section_6_gamma_2}
\end{equation}
\end{itemize}

where $C$ and $c$ are two positive constants.  
\end{theorem}
The proof of Theorem \ref{Main_Theorem_2} is based on some pointwise estimates in the Fourier space (Lemma \ref{Fourier_semigroup_estimates_section_4_gamma_0}) and  given in Subsection \ref{Subsection_Proof_1}. The optimality of the decay rates is given in Subsection \ref{Section_Optimality_1}.

\subsection{Pointwise estimates in the Fourier space}\label{Seb_Section_2_1}
In this section, we use the Lyapunov functional method to   show Theorem \ref{Main_Theorem_2}. We will prove later that the Lyapunov functional method agrees  with the   eigenvalues expansion (See Proposition \ref{Propo_1}) which shows the optimality of the estimates given in Theorem \ref{Main_Theorem_2}.

System (\ref{Fourier_system}) can be rewritten into the following form%
\begin{equation}
\left\{
\begin{array}{ll}
\hat{v}_{t}-i\xi \hat{u}+\hat{y}+l\hat{\eta}=0,\vspace{0.2cm} &  \\
\hat{u}_{t}-i\xi \hat{v}-lk\hat{\phi}=0,\vspace{0.2cm} &  \\
\hat{z}_{t}-ai\xi \hat{y}=0,\vspace{0.2cm} &  \\
\hat{y}_{t}-ai\xi \hat{z}-\hat{v}+\gamma _{1}\hat{y}=0,\vspace{0.2cm} &  \\
\hat{\phi}_{t}-i\xi k\hat{\eta}+lk_{0}\hat{u}=0,\vspace{0.2cm} &  \\
\hat{\eta}_{t}-i\xi k\hat{\phi}-l\hat{v}+\gamma _{2}\hat{\eta}=0. &
\end{array}%
\right.  \label{Fourier_system_2}
\end{equation}%
Let us now define the following energy functional%
\begin{equation}
\hat{E}\left(\xi, t\right) =\frac{1}{2}(\left\vert \hat{v}\right\vert
^{2}+\left\vert \hat{u}\right\vert ^{2}+\left\vert \hat{z}\right\vert
^{2}+\left\vert \hat{y}\right\vert ^{2}+|\hat{\phi}|^{2}+\left\vert \hat{\eta%
}\right\vert ^{2})(\xi,t).  \label{First_energy_seond_sound}
\end{equation}%
The next lemma states that the energy $\hat{E}\left( \xi,t\right) $ of
the entire system (\ref{Fourier_system_2}) (or equivalently system (\ref%
{Fourier_system})) is a non-increasing function. More precisely we have the
following result.

\begin{lemma}
\label{diss_Lemma} Let $(\hat{v},\hat{u},\hat{z},\hat{y},\hat{\phi},\hat{\eta%
})$ be the solution of (\ref{Fourier_system}), then the energy $%
\hat{E}\left( \xi,t\right) $ is a non-increasing function and satisfies,
for all $t\geq 0$,
\begin{equation}
\dfrac{d\hat{E}(\xi,t)}{dt}=-\gamma _{1}\left\vert \hat{y}\right\vert
^{2}-\gamma _{2}\left\vert \hat{\eta}\right\vert ^{2}.
\label{estimate derive}
\end{equation}
\end{lemma}
\begin{proof}
Multiplying the first equation in (\ref{Fourier_system_2})
by $\bar{\hat{v}}$, the second equation by $\bar{\hat{u}}$, the third
equation by $\bar{\hat{z}}$, the fourth equation by $\bar{\hat{y}}$, the
fifth equation $\bar{\hat{\phi}}$, the sixth equation by $\bar{%
\hat{\eta}}$, adding these equalities and taking the real part, then (%
\ref{estimate derive}) holds. 
\end{proof}
The following lemma is crucial for  the proof of Theorem \ref{Main_Theorem_2}. With this lemma in hand, we can show the decay estimates of the solution. 
\begin{lemma}
\label{Fourier_semigroup_estimates_section_4_gamma_0} Let $\hat{U}\left( \xi ,t\right) $ be
the solution of (\ref{Fourier_system}). Then for any $t\geq 0$ and $\xi \in
\mathbb{R}$, we have the following pointwise estimates:
\begin{equation}
\vert \hat{U}\left( \xi ,t\right) \vert ^{2}\leq \left\{
\begin{array}{ll}
Ce^{-c\rho _{1}\left( \xi \right) t}\vert \hat{U}\left( \xi ,0\right)
\vert ^{2}, & \text{ if}\quad a=1\vspace{0.3cm} \\
Ce^{-c\rho _{2}\left( \xi \right) t}\vert \hat{U}\left( \xi ,0\right)
\vert ^{2}, & \text{ if}\quad a\neq 1,
\end{array}%
\right.  \label{estimate_fourier_semigroup1_gamma_0_gamma_0}
\end{equation}
where
\begin{eqnarray}\label{kappa_1_gamma_0_gamma_0}
\rho _{1}\left( \xi \right) =\frac{\xi ^{2}}{1+\xi ^{2} }
\end{eqnarray}
 and 
 \begin{equation}\label{kappa_2_gamma_0}
\rho _{2}\left( \xi \right) =\frac{\xi ^{2}}{ 1+\xi ^{2}+\xi^4 }.
\end{equation}
 Here $C$ and $c$ are two positive constants.
\end{lemma}
\begin{proof}
The proof is based on a  delicate Fourier energy method. We do the proof in two main steps. 
\begin{description}
\item[ Step 1. Exhibiting dissipation of the other  terms] 
\end{description}
As we have seen from the estimate \eqref{estimate derive}, only two components of the solution  are damped through the energy dissipation inequality. So, our main goal in this first step is to find some appropriate functionals, that  give some dissipative terms to the other components of the vector solution. Indeed, 
we have from \cite{Said_Soufyane_2014_1}, the identity 
\begin{eqnarray}\label{dF_7_dt}
&&\frac{d}{dt}\mathscr{F}(\xi,t)+a^2l^2\xi
^{2}\left( \left\vert \hat{z}\right\vert ^{2}-\left\vert \hat{y}\right\vert
^{2}\right)-\xi^2\vert \hat{y}\vert^2+\xi^2\vert \hat{v}\vert^2\notag\\
&=&-al^2\gamma _{1}Re\left( i\xi \bar{\hat{z}}%
\hat{y}\right)+\gamma_2alRe(i\xi\hat{\eta}\bar{\hat{z}})+(1-a^2)l\xi^2Re(\hat{y}\bar{\hat{\eta}})\notag\\
&&+(a^2-1)Re(i\xi^3\hat{u}\bar{\hat{y}}),
\end{eqnarray}
where 
\begin{equation}\label{F_7_function}
\mathscr{F}(\xi,t):=la\left\{ lRe\left( i\xi \hat{y}\bar{\hat{z}}\right)+Re(i\xi\hat{z}\bar{\hat{\eta}})\right\} -\xi^2\left\{Re(\hat{v}\bar{\hat{y}})+Re(a\bar{\hat{z}}\hat{u})\right\}. 
\end{equation}
Also, from  \cite{Said_Soufyane_2014_1}, we have 
\begin{eqnarray}\label{dF_8_dt}
&&\frac{d}{dt}\mathscr{K}(\xi,t)+k_0\xi ^{2}(\vert \hat{\phi}\vert ^{2}-|\hat{\eta}|^{2})\notag\\
&=&
Re(i\xi lk_0\hat{u}%
\bar{\hat{\eta}})-\gamma_2Re(i\xi\hat{\eta}\bar{\hat{\phi}})
+Re%
(l a\xi ^{2}\hat{z}\bar{\hat{\phi}}) +lRe( \gamma _{1}i\xi \bar{\hat{\phi}}%
\hat{y})   \notag \\
&&-lk_{0}\xi ^{2}Re( \hat{\eta}\bar{\hat{y}}) -Re%
\left( l^2k_{0}i\xi \bar{\hat{y}}\hat{u}\right),
\end{eqnarray}
where 
\begin{equation}\label{F_8_Functional}
\mathscr{K}(\xi,t)= Re(- i\xi \hat{\phi}\bar{\hat{\eta}})+lRe (-i\xi \hat{y}\bar{\hat{\phi}}). 
\end{equation}
A simple application of Young's inequality gives
\begin{eqnarray}\label{dF_8_dt_main}
&&\frac{d}{dt}\mathscr{K}(\xi,t)+(k_0-\epsilon_1)\xi ^{2}\vert \hat{\phi}\vert ^{2}\notag\\
&\leq & C(\epsilon_1,\epsilon_1^\prime)(1+\xi^2)(\vert\hat{\eta}\vert^{2}+\vert\hat{y}\vert^{2})+C(\epsilon_1)\xi^2\vert\hat{z}\vert^{2}+\epsilon_1^\prime \xi^2\vert\hat{u}\vert^{2},
\end{eqnarray}
where $\epsilon_1,\epsilon_1^\prime$ are  arbitrary small positive constants. 

Concerning \eqref{dF_7_dt}, we have the following estimates:
 \begin{itemize}
\item For $a=1$, we have as above, for any $\epsilon_2,\epsilon_2^\prime$ positive:
\begin{eqnarray}\label{dF_7_dt_main_1}
&&\frac{d}{dt}\mathscr{F}(\xi,t)+(a^2l^2-\epsilon_2)\xi
^{2} \left\vert \hat{z}\right\vert ^{2}
+\xi^2\vert \hat{v}\vert^2\notag\\
&\leq&C(\epsilon_2)(1+\xi^2)(\left\vert \hat{\eta}\right\vert ^{2}+\left\vert \hat{y}\right\vert ^{2}).
\end{eqnarray}
\item For $a\neq 1$, we obtain, instead of \eqref{dF_7_dt_main_1},
\begin{eqnarray}\label{dF_7_dt_main_2}
&&\frac{d}{dt}\mathscr{F}(\xi,t)+(a^2l^2-\epsilon_2)\xi
^{2} \left\vert \hat{z}\right\vert ^{2}
+\xi^2\vert \hat{v}\vert^2\notag\\
&\leq&C(\epsilon_2,\epsilon_2^\prime)(1+\xi^2+\xi^4)(\left\vert \hat{\eta}\right\vert ^{2}+\left\vert \hat{y}\right\vert ^{2})+\epsilon_2^\prime\xi^2\left\vert \hat{u}\right\vert ^{2},
\end{eqnarray}
where we used the estimate:
\begin{eqnarray*}
\left\vert(a^2-1)Re(i\xi^3\hat{u}\bar{\hat{y}})\right\vert \leq \epsilon_2^\prime\xi^2\vert \hat{u}\vert^2+C(\epsilon_2^\prime)\xi^4\vert \hat{y}\vert^2.
\end{eqnarray*}
\end{itemize}

Next, multiplying the first equation in (\ref{Fourier_system_2}) by $i\xi 
\bar{\hat{u}}$, the second equation by $-i\xi \bar{\hat{v}}$, 
 adding the results and  taking the real part, we get 
\begin{eqnarray}
&&\frac{d}{dt}Re\left( i\xi \hat{v}\bar{\hat{u}}\right) +\xi ^{2}(\left\vert \hat{u}\right\vert
^{2}-\left\vert \hat{v}\right\vert ^{2})\notag\\
&&=-Re\left( i\xi \hat{y}\bar{%
\hat{u}}\right)  
-Re\left( i\xi l\hat{\eta}\bar{\hat{u}}\right)-Re%
(i\xi k_0 l\hat{\phi}\bar{\hat{v}}) .
\label{first_term_main_new_1}
\end{eqnarray}

Multiplying the first equation in \eqref{Fourier_system_2} by $-\bar{\hat{\eta}}
$ and the sixth equation by $-\bar{\hat{v}}$, then taking the real
part after adding the two results, we obtain
\begin{eqnarray}\label{Estimate_Main_Term_new_gamma_0_gamma_2}
&&-\frac{d}{dt}Re\left( \hat{v}%
\bar{\hat{\eta}}\right)+l\left\vert \hat{v}\right\vert ^{2}-l\left\vert \hat{\eta}\right\vert ^{2}\notag\\
&&=-Re(i\xi \bar{\hat{\eta}}\hat{u})+%
Re(\hat{y}\bar{\hat{\eta}})-Re(i\xi k_0\hat{\phi}\bar{\hat{v}})+\gamma_2Re( \bar{\hat{v}}\hat{\eta}). 
\end{eqnarray}
Summing up $\eqref{first_term_main_new_1}+l \eqref{Estimate_Main_Term_new_gamma_0_gamma_2}$, we get 
\begin{eqnarray*}
&&\frac{d}{dt}\mathscr{P}(\xi,t)+\xi ^{2}(\left\vert \hat{u}\right\vert
^{2}-\left\vert \hat{v}\right\vert ^{2})+l^2\left\vert \hat{v}\right\vert ^{2}-l^2\left\vert \hat{\eta}\right\vert ^{2}\notag\\
&=&-Re\left( i\xi \hat{y}\bar{%
\hat{u}}\right)  
-2Re%
(i\xi k_0 l\hat{\phi}\bar{\hat{v}})+%
lRe(\hat{y}\bar{\hat{\eta}})+l\gamma_2Re( \bar{\hat{v}}\hat{\eta}), 
\end{eqnarray*}
where 
\begin{equation}\label{P_Functional}
\mathscr{P}(\xi,t)=Re\left( i\xi \hat{v}\bar{\hat{u}}\right)-lRe\left( \hat{v}%
\bar{\hat{\eta}}\right). 
\end{equation}
Applying Young's inequality, we obtain for any $\epsilon_3,\epsilon_4>0$, 
\begin{eqnarray}\label{dP_dt_Main}
&&\frac{d}{dt}\mathscr{P}(\xi,t)+\xi ^{2} (1-\epsilon_3)\left\vert \hat{u}\right\vert
^{2}+(l^2-\epsilon_4)\left\vert \hat{v}\right\vert ^{2}\notag\\
&\leq &\xi^2\left\vert \hat{v}\right\vert ^{2} +C(\epsilon_3,\epsilon_4) (|\hat{y}|^2+|\hat{\eta}|^2) +C(\epsilon_4)\xi^2|\hat{\phi}|^2. 
\end{eqnarray}
\begin{description}
\item[Step 2.  Building the appropriate Lyapunov functional] 
\end{description}
In this step, we make the appropriate combination of the above obtained functionals to build a Lyapunov functional $\mathscr{L}(\xi,t)$. To construct  this functional, we need to take into account two main things.  First, this functional should satisfy the estimate \eqref{L_6_main} and  second, it should verify  another estimate of the form 
\begin{equation*}
c_1\sigma (\xi) \hat{E}(\xi,t)\leq \mathscr{L}(\xi,t)\leq c_2\sigma (\xi) \hat{E}(\xi,t),
\end{equation*}
where $c_1$ and $c_2$ are two positive constants and $\sigma(\xi)$ is a function depending on $\xi$ only. 

Hence, we define for 
$a=1$,  the Lyapunov functional $\mathscr{L}_1(\xi,t)$ as follows:
\begin{equation}\label{Z_6_function_gamma_0}
\mathscr{L}_1(\xi,t):=d_0(1+\xi^2)\hat{E}(\xi,t)+d_1\mathscr{F}(\xi,t)+d_2\mathscr{K}(\xi,t)+\mathscr{P}(\xi,t),
\end{equation}
where $d_0, d_1, d_2$ and $d_3$ are  positive constants that will be fixed later. 

The derivative of \eqref{Z_6_function_gamma_0} with respect to $t$ and the use of \eqref{estimate derive}, \eqref{dF_8_dt_main}, \eqref{dF_7_dt_main_1} and \eqref{dP_dt_Main} lead to 
\begin{eqnarray}\label{dL_6_dt}
&&\frac{d}{dt}\mathscr{L}_1(\xi,t)+\Big\{d_1(a^2l^2-\epsilon_2)-d_2C(\epsilon_1)\Big\}\xi
^{2} \left\vert \hat{z}\right\vert ^{2}\notag\\
&&+(d_1-1)\xi^2\vert \hat{v}\vert^2+\Big\{d_2(k_0-\epsilon_1)-C(\epsilon_4)\Big\}\xi ^{2}\vert \hat{\phi}\vert ^{2}\notag\\
&&+\Big\{(1-\epsilon_3)-d_2\epsilon_1^\prime\Big\}\xi^2\vert\hat{u}\vert^{2}+(l^2-\epsilon_4)|\hat{v}|^2
\notag\\
&\leq&\Big\{C(\epsilon_1,\epsilon_1^\prime,\epsilon_2,\epsilon_3,\epsilon_4, d_1, d_2)-d_0\min(\gamma_1,\gamma_2)\Big\}(1+\xi^2)(\left\vert \hat{\eta}\right\vert ^{2}+\left\vert \hat{y}\right\vert ^{2}).
\end{eqnarray}
We choose the constants in the above formula as follows: fix $\epsilon_1,\epsilon_2$ and $\epsilon_3$ small enough such that 
\begin{equation*}
\epsilon_1<k_0,\qquad \epsilon_2<a^2l^2,\qquad \epsilon_3<1,\qquad \epsilon_4<l^2. 
\end{equation*}
After that, we fix $d_2$ large enough such that 
\begin{equation*}
d_2>\frac{C(\epsilon_4)}{k_0-\epsilon_1}. 
\end{equation*}
Then, we select $d_1$ large enough such that 
\begin{equation*}
d_1>\max\left(1,\frac{d_2C(\epsilon_1)}{a^2l^2-\epsilon_2}\right). 
\end{equation*}
Furthermore, we pick $\epsilon_1^\prime$ small enough such that 
\begin{equation*}
\epsilon_1^\prime<\frac{1-\epsilon_3}{d_2}.
\end{equation*}
Finally, once all the above constants are fixed, we take $d_0$ large enough such that 
\begin{equation*}
d_0>\frac{C(\epsilon_1,\epsilon_1^\prime,\epsilon_2,\epsilon_3,\epsilon_4, d_1, d_2)}{\min(\gamma_1,\gamma_2)}. 
\end{equation*}
Hence,  we find a positive constant $c_0>0$, such that 
\begin{eqnarray}\label{L_6_main}
\frac{d}{dt}\mathscr{L}_1(\xi,t)+c_0\xi^2 \hat{E}(\xi,t)\leq 0,\qquad \forall t\geq0.
\end{eqnarray}
Since  
\begin{equation}\label{Equiv_U_E}
\hat{E}(\xi,t)=\frac{1}{2}\vert \hat{U}(\xi,t)\vert^2
\end{equation}
then,  for $d_0$  large enough,  there exist two positive  constants $c_1$ and $c_2$ such that for all $t\geq0$,
\begin{equation}\label{Equiv_Z_E}
c_1(1+\xi^2) \hat{E}(\xi,t)\leq \mathscr{L}_1(\xi,t)\leq c_2(1+\xi^2) \hat{E}(\xi,t).
\end{equation}
On the other hand,    there exists a constant $c_3>0$, such that 
\begin{eqnarray}\label{L_6_main_2}
\frac{d}{dt}\mathscr{L}_1(\xi,t)+c_3\frac{\xi^2 }{1+\xi^2}\mathscr{L}_1(\xi,t)\leq 0,\qquad \forall t\geq0.
\end{eqnarray}
Integrating \eqref{L_6_main_2} and using once again \eqref{Equiv_U_E} and \eqref{Equiv_Z_E}, then \eqref{estimate_fourier_semigroup1_gamma_0_gamma_0} holds for $a=1$. 


Next, for $a\neq 1$, we define  another  Lyapunov Functional 
\begin{equation}\label{Z_7_function_gamma_0}
\mathscr{L}_2(\xi,t)=d_0(1+\xi^2+\xi^4)\hat{E}(\xi,t)+d_1\mathscr{F}(\xi,t)+d_2\mathscr{K}(\xi,t)+\mathscr{P}(\xi,t).
\end{equation}
Now, arguing as above and choosing the constants exactly as before, except for the new constant $\epsilon_2^\prime$ which should be small enough, we get 
\begin{eqnarray}\label{L_7_main}
\frac{d}{dt}\mathscr{L}_2(\xi,t)+c_4\frac{\xi^2}{(1+\xi^2+\xi^4)}\hat{E}(\xi,t)\leq 0,\qquad \forall t\geq0,
\end{eqnarray}
for some $c_4>0$. Which leads to the second estimate in \eqref{estimate_fourier_semigroup1_gamma_0_gamma_0}. We omit the details. 
\end{proof}
\subsection{Decay estimates: Proof of Theorem \ref{Main_Theorem_2} }\label{Subsection_Proof_1}
In this subsection, we show the decay estimate of the $L^2$-norm of the solution of \eqref{First_order_system}.  These decay estimates are optimal,  since  they agree with the asymptotic expansion of the eigenvalues given in Subsection \ref{Section_Optimality_1}. In addition, Theorem \ref{Main_Theorem_2} improves the result of Theorem 6.1 in \cite{Said_Soufyane_2014_1}. 

To  show \eqref{Decay_U_estimate_L_2_1_gamma_0_1_gamma_2}, we have from   \eqref{kappa_1_gamma_0_gamma_0}   that
\begin{equation}\label{rho_2_behavior_new}
\rho_1(\xi)\geq\left\{
\begin{array}{ll}
    c\xi^2, & \text{if } \xi\leq 1, \vspace{0.2cm} \\
   c,  &   \text{if } \xi\geq 1.
\end{array}
\right.
\end{equation} 
Applying the Plancherel theorem and using the first estimate in  (\ref%
{estimate_fourier_semigroup1_gamma_0_gamma_0}), we obtain%
\begin{eqnarray}
\left\Vert \partial _{x}^{j}U\left( t\right) \right\Vert _{L^{2}}^{2}
&=&\int_{\mathbb{R} }\vert \xi \vert ^{2j}\vert \hat{U}%
\left( \xi ,t\right) \vert ^{2}d\xi  \notag \\
&\leq &C\int_{\mathbb{R} }\left\vert \xi \right\vert ^{2j}e^{-c\rho_1 \left(
\xi \right) t}\vert \hat{U}\left( \xi ,0\right) \vert ^{2}d\xi
\notag \\
&=&C\int_{\left\vert \xi \right\vert \leq 1}\left\vert \xi \right\vert
^{2j}e^{-c\rho_1 \left( \xi \right) t}\vert \hat{U}\left( \xi ,0\right)
\vert ^{2}d\xi +C\int_{\left\vert \xi \right\vert \geq 1}\left\vert
\xi \right\vert ^{2j}e^{-c\rho_1 \left( \xi \right) t}\vert \hat{U}\left(
\xi ,0\right) \vert ^{2}d\xi  \notag \\
&:=&I_{1}+I_{2}.  \label{deron_U_equality}
\end{eqnarray}%
Exploiting (\ref{rho_2_behavior_new}),
 we infer that
\begin{equation}
I_{1}\leq C\Vert \hat{U}_{0}\Vert _{L^\infty }^{2}\int_{\left\vert
\xi \right\vert \leq 1}\left\vert \xi \right\vert ^{2j}e^{-c\xi
^{2}t}d\xi \leq C\left( 1+t\right) ^{-\frac{1}{2}\left( 1+2j\right)
}\left\Vert U_{0}\right\Vert _{L^{1}}^{2},  \label{I_1_estimate}
\end{equation}
where we have used the inequality%
\begin{equation}\label{Integral_Inequality}
\int_{0}^{1}\left\vert \xi \right\vert ^{\sigma }e^{-c\xi ^{2}t}d\xi \leq
C\left( 1+t\right) ^{-\left( \sigma +1\right) /2}.
\end{equation}%
In the high frequency region ($|\xi|\geq 1$), we have 
\begin{eqnarray*}
I_2&\leq& e^{-ct}\int_{\left\vert \xi \right\vert \geq 1}\left\vert \xi
\right\vert ^{2j}\vert \hat{U}\left( \xi ,0\right)
\vert ^{2}d\xi\\
&\leq& e^{-ct} \Vert \partial_x^jU_0\Vert^2_{L^2}
\end{eqnarray*}
which leads to the estimates in \eqref{Decay_U_estimate_L_2_1_gamma_0_1_gamma_2}. 

Second, assume that $a\neq 1$. As above,%
\begin{eqnarray}
\left\Vert \partial _{x}^{j}U\left( t\right) \right\Vert _{L^2}^{2}
&\leq&C\int_{\left\vert \xi \right\vert \leq 1}\left\vert \xi \right\vert
^{2j}e^{-c\rho _{2}\left( \xi \right) t}\vert \hat{U}\left( \xi
,0\right) \vert ^{2}d\xi +C\int_{\left\vert \xi \right\vert \geq
1}\left\vert \xi \right\vert ^{2j}e^{-c\rho _{2}\left( \xi \right)
t}\vert \hat{U}\left( \xi ,0\right) \vert ^{2}d\xi  \notag \\
&=&L_{1}+L_{2}.  \label{Main_identity_a_2}
\end{eqnarray}%
Now using the second estimate in (\ref{estimate_fourier_semigroup1_gamma_0_gamma_0}) and the
fact that $\rho _{2}(\xi )\geq \frac{1}{4}\xi ^{2}$ for $\left\vert \xi
\right\vert \leq 1$, we have by the same method as in the proof of the
estimate of $I_{1},$%
\begin{equation}
L_{1}\leq C\left( 1+t\right) ^{-1/2-j}\left\Vert U_{0}\right\Vert
_{L^{1}}^{2}.  \label{L_1_estimate}
\end{equation}%
To estimate the term $L_{2}$, we use the inequality $\rho _{2}(\xi )\geq
c\xi ^{-2}$ for $\left\vert \xi \right\vert \geq 1$ to obtain
\begin{eqnarray}
L_{2}
&\leq &C\sup_{\left\vert \xi \right\vert \geq 1}\left( \left\vert \xi
\right\vert ^{-2\ell}e^{-c\xi ^{-2}t}\right) \int_{\left\vert \xi \right\vert
\geq 1}\left\vert \xi \right\vert ^{2\left( j+\ell\right) }\vert \hat{U}%
\left( \xi ,0\right) \vert ^{2}d\xi  \notag \\
&\leq &C\left( 1+t\right) ^{-\ell}\left\Vert \partial
_{x}^{j+\ell}U_{0}\right\Vert _{L^2}^{2}.  \label{L_2_estimate}
\end{eqnarray}%
Inserting the estimates (\ref{L_1_estimate}) and (\ref{L_2_estimate}) into (%
\ref{Main_identity_a_2}), then (\ref{Decay_U_estimate_L_2_1_section_6_gamma_2}) is obtained.
This finishes the proof of Theorem \ref{Main_Theorem_2}.

\subsection{Optimality of the decay rates}\label{Section_Optimality_1}
To prove the optimality of the decay rate in  Theorem \ref{Main_Theorem_2}, we use the following proposition based on  eigenvalues expansion.

\begin{proposition}\label{Propo_1}
Let $\lambda_j(\zeta),\, 1\leq j\leq 6$ be the eigenvalues of $\Phi(\zeta)$.
Then as $|\xi|\rightarrow0$ 
\begin{equation}\label{Estimate_Eigenvalues_0}
Re(\lambda_j)(i\xi)=
\left\{
\begin{array}{ll}
-\dfrac{a^2l^2}{\gamma_1l^2+\gamma_2}|\xi|^2+O(|\xi|^3),&\qquad \text{for}\qquad j=1,\vspace{0.3cm}\\
-Re(\beta_j)|\xi|^2+O(|\xi|^3),&\qquad \text{for}\qquad j=2,3,\vspace{0.3cm}\\
Re(r_j)+O(|\xi|),&\qquad \text{for}\qquad j=4,5,6.
\end{array}
\right. 
\end{equation}

For  $|\xi|\rightarrow+\infty$

\begin{itemize}
\item For $a=1$, we get 
\begin{equation}\label{Estimate_Eigenvalues_infinity_1}
Re(\lambda_j)(i\xi)=
\left\{
\begin{array}{ll}
Re(\delta_j)+O(|\xi|^{-1}),&\qquad \text{for}\qquad j=1,2.
\vspace{0.2cm}\\
-\dfrac{\gamma_1}{2}+O(|\xi|^{-1}),&\qquad \text{for}\qquad j=3,4,\vspace{0.2cm}\\
-\dfrac{\gamma_2}{2}+O(|\xi|^{-1}),&\qquad \text{for}\qquad j=5,6.
\end{array}
\right. 
\end{equation}
\item For $a\neq 1$, we have 
\begin{equation}\label{Estimate_Eigenvalues_infinity_2}
Re(\lambda_j)(i\xi)=
\left\{
\begin{array}{ll}
-\kappa_j|\xi|^{-2}+O(|\xi|^{-3}),&\qquad \text{for}\qquad j=1,2\vspace{0.2cm}\\
-\dfrac{\gamma_1}{2}+O(|\xi|^{-1}),&\qquad \text{for}\qquad j=3,4,\vspace{0.2cm}\\
-\dfrac{\gamma_2}{2}+O(|\xi|^{-1}),&\qquad \text{for}\qquad j=5,6.
\end{array}
\right. 
\end{equation}
\end{itemize}

\end{proposition}

\begin{proof}

For $\gamma_i> 0,\, i=1,2$, the  system (\ref{Bresse_main_system_2})-(\ref{Initial_data_2}) is equivalent
to  (\ref{First_order_system}) with%
\begin{equation*}
U=\left(
\begin{array}{c}
v \\
u \\
z \\
y \\
\phi \\
\eta%
\end{array}%
\right) ,\,\ A=\left(
\begin{array}{cccccc}
0 & -1 & 0 & 0 & 0 & 0 \\
-1 & 0 & 0 & 0 & 0 & 0 \\
0 & 0 & 0 & -a & 0 & 0 \\
0 & 0 & -a & 0 & 0 & 0 \\
0 & 0 & 0 & 0 & 0 & -k \\
0 & 0 & 0 & 0 & -k & 0%
\end{array}%
\right) ,\,\ L=\left(
\begin{array}{cccccc}
0 & 0 & 0 & 1 & 0 & l \\
0 & 0 & 0 & 0 & -lk & 0 \\
0 & 0 & 0 & 0 & 0 & 0 \\
-1 & 0 & 0 & \gamma_1  & 0 & 0 \\
0 & lk & 0 & 0 & 0 & 0 \\
-l & 0 & 0 & 0 & 0 & \gamma_2%
\end{array}%
\right)
\end{equation*}%
and $U_{0}=\left( v_{0},u_{0},z_{0},y_{0},\phi _{0},\eta _{0}\right) ^{T}$. Observe that, since $L$ is not symmetric, the general theory for hyperbolic systems does not apply. 

Recall that 
\begin{equation}\label{Phi_Matrix}
\Phi(\zeta)=-(L+\zeta A),\qquad \zeta=i\xi\in \C. 
\end{equation}
Let us denote by $\lambda_j(\zeta),\, 1\leq j\leq 6$ the eigenvalues of $\Phi(\zeta)$, then we  compute the characteristic equation as 
\begin{eqnarray}\label{Characteristic_Poly_2}
\det{\left(\lambda I-\Phi(\zeta)\right)}&=&\lambda^6+(\gamma_1+\gamma_2)\lambda^5+\Big\{(k^2+1)(l^2-\zeta ^2) +\gamma_1\gamma_2+1-a^2\zeta^2\Big\}\lambda^4\notag\\
&&+\Big\{ \gamma_1(k^2+1)(  l^2-\zeta ^2 )+\gamma_2\left((k^2l^2+1)-(1+a^2)\zeta^2\right)\Big\}\lambda^3\notag\\
&&+\left[\gamma_1\gamma_2(k^2l^2-\zeta^2)+(l^2-\zeta^2) \Big\{k^2(l^2-\zeta^2)+\Big(k^2-a^2(k^2+1)\zeta^2\Big)\Big\}\right]\lambda^2\\
&&+\Big\{\gamma_1 k^2( l^2-\zeta ^2)^2+ k^2 l^2 \gamma_2  - a^2 k^2 l^2 \gamma_2 \zeta^2  
+ a^2 \gamma_2 \zeta^4\Big\}\lambda-a^2k^2\zeta^2(l^2-\zeta^2)^2.\notag
\end{eqnarray}
It is legitimate to do an asymptotic expansion of the eigenvalues.
Indeed, since the dependence on $\zeta$  of $\Phi$ is analytic then  by \cite[p.63]{Kat76_2} the eigenvalues depends  also analyticly on $\zeta$. 
\begin{itemize}
\item \textbf{Behavior of $\lambda_j(\zeta)$
 when $|\zeta|\rightarrow 0$.}
\end{itemize}

First, when $|\zeta|\rightarrow 0$, then $\lambda_j(\zeta)$ has the following asymptotic expansion: 
\begin{equation}\label{Asymptotic_Expansion_2}
\lambda_j(\zeta)=\lambda_j^{(0)}+\lambda_j^{(1)}\zeta+\lambda_j^{(2)}\zeta^2+...,\qquad \qquad 1\leq j\leq 6.
\end{equation}
Notice that $\lambda_j^{(0)}$ are the eigenvalues of the matrix $-L$ and satisfy, 
with $y=\lambda_j^{(0)}$, the equation \begin{equation*}
y\left(y^2+k^2 l^2\right) \left(y^3+(\gamma_1+\gamma_2) y^2+(l^2+1+\gamma_1\gamma_2)y+\gamma_1 l^2+\gamma_2\right)=0. 
\end{equation*}
Consequently,  we have from the above equation that 
\begin{equation*}
\left\{
\begin{array}{ll}
\lambda_j^{(0)}=0,&\qquad \text{for}\qquad j=1,\vspace{0.2cm}\\
\lambda_j^{(0)}=\pm ikl,&\qquad \text{for}\qquad j=2,3,\vspace{0.2cm}\\
\lambda_j^{(0)}=r_j,&\qquad \text{for}\qquad j=4,5,6,
\end{array}
\right. 
\end{equation*}
where $r_j$ are the solutions of the algebraic equation 
\begin{equation}\label{Equation_1_Roots_2}
g(X)=X^3+(\gamma_1+\gamma_2) X^2+(l^2+1+\gamma_1+\gamma_2)X+\gamma_1 l^2+\gamma_2=0. 
\end{equation}
It is well known that an algebraic equation of an odd degree with real coefficients has at least
one real root $r_1$.  Now, in order to know the location of $r_1$, we consider the equation  \eqref{Equation_1_Roots_2} with $X\in \R$. 
Then, it is clear that 
\begin{equation*}
g(-(\gamma_1+\gamma_2))g(0)=-(\gamma_1l^2+\gamma_2)\left(\gamma_1^2 + \gamma_2 (l^2 +\gamma_2) + \gamma_1 (1 + 2 \gamma_2)\right)<0. 
\end{equation*}
Therefore,  equation \eqref{Equation_1_Roots_2} has at least one real root $X=r_1$ in the interval $(-(\gamma_1+\gamma_2),0)$. In this case, we may rewrite equation \eqref{Equation_1_Roots_2} in the form 
\begin{equation}\label{Poly_Equation_2}
g(X)=(X-r_1)\left(X^2+(\gamma_1+\gamma_2+r_1)X+l^2+1+\gamma_1+\gamma_2+(\gamma_1+\gamma_2)r_1+r_1^2\right).
\end{equation}
Now, let us  denote by $r_2$ and $r_3$, the other two roots. Then, we have 
\begin{equation*}
r_1+r_2+r_3=-(\gamma_1+\gamma_2),
\end{equation*}
and
\begin{equation*}
r_1r_2r_3=-(\gamma_1l^2+\gamma_2).
\end{equation*}
Since $r_1$ is a real root, then the coefficients of \eqref{Poly_Equation_2} are real and therefore $$Re(r_2)=Re(r_3).$$ 
This implies that 
\begin{equation*}
Re(r_2)=Re(r_3)=-\frac{1}{2}(r_1+\gamma_1+\gamma_2)<0.
\end{equation*}
If $r_2$ and $r_3$ are real, then they satisfy
$$r_2+r_3=-(\gamma_1+\gamma_2+r_1)<0, \quad \mbox{ and }\quad  r_2r_3=-\frac{(\gamma_1l^2+\gamma_2)}{r_1}>0,$$  which implies $r_2,r_3<0.$

Now, using equation \eqref{Characteristic_Poly_2} and \eqref{Asymptotic_Expansion_2}, by equating coefficients  of like powers of $\zeta$, we obtain 
\begin{equation*}
\left\{
\begin{array}{ll}
\lambda_j^{(1)}=0,&\qquad \text{for}\qquad j=1,2,3\vspace{0.2cm}\\
\lambda_j^{(2)}=\dfrac{a^2l^2}{\gamma_1l^2+\gamma_2},&\qquad \text{for}\qquad j=1,\vspace{0.2cm}\\
\lambda_j^{(2)}= \beta_j,&\qquad \text{for}\qquad j=2,3,
\end{array}
\right. 
\end{equation*}
where $\beta_j$ is the solution  of the equation 
$$A+iB+\beta_j (C+iD)=0,$$
where 
\begin{equation*}
\left\{
\begin{array}{ll}
A=k^2 l^2 \Big(k^2(1+l^2)  - k^4 l^2 + \gamma_1 \gamma_2\Big)\vspace{0.3cm}\\
B=k^3 l^3 \Big(\gamma_1 (k^2-1)+ \gamma_2\Big)\vspace{0.3cm}\\
C=2 k^2 l^2 \Big(\gamma_1 l^2(k^2-1) +\gamma_2(k^2l^2-1) \Big)\vspace{0.3cm}\\
D=2 k^3 l^3 \Big( l^2(k^2-1)  -1  - \gamma_1 \gamma_2\Big). 
\end{array}
\right. 
\end{equation*}
We need to show that $Re(\beta_j)>0.$ In order to prove this,  it is enough to verify  that 
\begin{equation*}
AC+BD<0. 
\end{equation*}
Hence,  
\begin{eqnarray*}
K=AC+BD&=&\Big(k^2(1+l^2)  - k^4 l^2 + \gamma_1 \gamma_2\Big)\Big(\gamma_1 l^2(k^2-1) +\gamma_2(k^2l^2-1) \Big)\\
&&+k^2 l^2 \Big(\gamma_1 (k^2-1)+ \gamma_2\Big)\Big( l^2(k^2-1)  -1  - \gamma_1 \gamma_2\Big). 
\end{eqnarray*}
Factorizing by $k^4l^4\gamma_2$, we deduce 
\begin{equation*}
K=-2 k^4 l^4 \gamma _2 \left(\gamma _1 \gamma _2+\left(k^2-1\right)^2 l^2 \gamma _1^{2  }+k^2 \left(\left(k^2-1\right) l^2-1\right)^2\right)<0,
\end{equation*}
which concludes the proof of \eqref{Estimate_Eigenvalues_0}
\begin{itemize}
\item \textbf{Behavior of $\lambda_j(\zeta)$
 when $|\zeta|\rightarrow \infty$.}
 \end{itemize}

For $|\zeta|\rightarrow \infty$ and  following \cite{IHK08},  we consider the characteristic equation in the form 
\begin{eqnarray}\label{Charact_Equation_mu}
\Psi(\zeta^{-1}) &=&\zeta^6\det {(\mu I+(A+\zeta^{-1} L))}\notag\vspace{0.3cm}\\
&=& \mu^6+(\gamma_1+\gamma_2)\zeta^{-1}\mu^5+\Big\{(k^2+1)(l^2\zeta^{-2}-1) +(\gamma_1\gamma_2+1)\zeta^{-2}-a^2\Big\}\mu^4\notag\\
&&+\Big\{ \gamma_1(k^2+1)(  l^2\zeta^{-2}-1 )+\gamma_2\left((k^2l^2+1)\zeta^{-2}-(1+a^2)\right)\Big\}\zeta^{-1}\mu^3\notag\\
&&+\left[\gamma_1\gamma_2(k^2l^2\zeta^{-2}-1)\zeta^{-2}+(l^2\zeta^{-2}-1) \Big\{k^2(l^2\zeta^{-2}-1)+\Big(k^2\zeta^{-2}-a^2(k^2+1)\Big)\Big\}\right]\mu^2\\
&&+\Big\{\gamma_1 k^2( l^2\zeta^{-2}-1)^2+ k^2 l^2 \gamma_2\zeta^{-4}  - a^2 k^2 l^2 \gamma_2 \zeta^{-2}  
+ a^2 \gamma_2 \Big\}\zeta^{-1}\mu-a^2k^2(l^2\zeta^{-2}-1)^2=0,\notag
\end{eqnarray}
where $\mu(\zeta^{-1})$ is the eigenvalues of 
\eqref{Charact_Equation_mu}.
Moreover, we  have the relation
\begin{equation*}
\lambda_j(\zeta)=\zeta\mu_j(\zeta^{-1}). 
\end{equation*}
Now, for $|\zeta|^{-1}\rightarrow 0$, we have the asymptotic expansion of $\mu_j(\zeta^{-1})$ in the form (for simplicity, we put $\nu=\zeta^{-1}$)
\begin{equation}\label{Asymptotic_Expansion_mu}
\mu_j(\nu)=\mu_j^{(0)}+\mu_j^{(1)}\nu+\mu_j^{(2)}\nu^{2}+...,\qquad \qquad 1\leq j\leq 6.
\end{equation}
 Plugging \eqref{Asymptotic_Expansion_mu} into \eqref{Charact_Equation_mu} and  equating coefficients  of like powers of $\nu$, we get

\begin{equation*}
\left\{
\begin{array}{ll}
\mu_j^{(0)}=\pm 1,&\qquad \text{for}\qquad j=1,2\vspace{0.2cm}\\
\mu_j^{(0)}=\pm a,\qquad \mu_j^{(1)}=-\dfrac{\gamma_1}{2}\qquad\qquad \qquad &\qquad \text{for}\qquad j=3,4,\vspace{0.2cm}\\
\mu_j^{(0)}=\pm k \qquad \mu_j^{(1)}=-\dfrac{\gamma_2}{2}&\qquad \text{for}\qquad j=5,6.
\end{array}
\right. 
\end{equation*}

 \begin{itemize}
\item For $a=1$, 
\begin{equation*}
\mu_j^{(1)}=\delta_j=\frac{1}{4} \left(-\gamma_1\pm \sqrt{\gamma_1^2-4}\right),\qquad \text{for}\qquad j=1,2
\end{equation*}
It is clear that $Re(\delta_j)<0$. 
\item For $a\neq 1$, we have 

\begin{equation*}
\left\{
\begin{array}{ll}
 \mu_j^{(1)}=0,\vspace{0.3cm}\\
 \mu_j^{(2)}=\pm\dfrac{l^2(1-a^2)+1}{2 \left(a^2-1\right)},\vspace{0.3cm}\\
 \mu_j^{(3)}=\dfrac{\left(a^2-1\right)^2 l^2 \gamma_2+\gamma_1}{2 \left(a^2-1\right)^2}=\kappa_j>0,
\end{array}
\qquad \text{for}\qquad j=1,2
\right. 
\end{equation*}
\end{itemize}
which concludes the proof of  Proposition \ref{Propo_1}. 
\end{proof}

\section{The case $\gamma_1=0$ and $\gamma_2>0$}\label{Section_5}
In this section, we investigate the case where $\gamma_1=0$ and $\gamma_2>0$. In this case the only acting  damping term  $\gamma_2w_t$ of the whole system is in the third equation. We prove that the effect of this damping term will be propagated to the other components  of the solution which will lead eventually, to the convergence of the solution to zero.  This requires to assume  more regularity on the initial data than the case where $\gamma_i>0,\, i=1,2$.  Moreover, our result has been proved under a new extra assumption on the coefficients of the system and it reads as  follows. 
\begin{theorem}\label{Theorem_One_Damping}
Assume that $\gamma_1=0$ and
\begin{equation*}
\left(k^2-1\right) l^2-1\neq 0.
\end{equation*}
  Let $s$ be a nonnegative integer and  $U(x,t)$ be the solution of \eqref{First_order_system}. Assume that $U_0\in H^s(\R)\cap L^1(\R)$,  then the following decay estimates hold for $t$ large enough:
\begin{itemize}
\item For $a=k=1$, we have 
\end{itemize}
\begin{equation}
\left\Vert \partial _{x}^{j}U\left( t\right) \right\Vert _{L^{2}}\leq
C\left( 1+t\right) ^{-1/4-j/2 }\left\Vert
U_{0}\right\Vert _{L^{1}}+C(1+t)^{-\ell/2}\left\Vert \partial
_{x}^{j+\ell+1}U_{0}\right\Vert _{L^{2}}, 
 \, 0\leq j\leq s-\ell-1. \label{Decay_U_estimate_L_2_1_gamma_0_1_gamma_2_k}
\end{equation}
\begin{itemize}
\item For $a\neq 1$, we have 
\end{itemize}
\begin{equation}
\left\Vert \partial _{x}^{j}U\left( t\right) \right\Vert _{L^{2}}\leq
C\left( 1+t\right) ^{-1/4-j/2 }\left\Vert
U_{0}\right\Vert _{L^{1}}+C\left( 1+t\right) ^{-\ell/2}\left\Vert \partial
_{x}^{j+\ell+3}U_{0}\right\Vert _{L^{2}},\, 0\leq j\leq s-\ell-3. 
\label{Decay_U_estimate_L_2_1_section_6_gamma_2_k}
\end{equation} 
\end{theorem}
The proof of Theorem \ref{Theorem_One_Damping} will be given in Subsection \ref{Theorem_One_Damping_1}. However, it seems difficult to build  appropriate Lyapunov functionals in this case, as we did in Section \ref{Seb_Section_2_1}. Instead, we rely on asymptotic expansion of the eigenvalues of the matrix $\Phi(\zeta)$ and on the behavior of the Fourier image of the solution in  low frequencies $\Upsilon_{L}=\left\{|\xi|<\nu\ll 1\right\}$, 
 middle frequencies $\Upsilon_{M}=\left\{\nu\leq|\xi|\leq N\right\}$ and high frequencies $\Upsilon_{H}=\left\{|\xi|>N\gg 1\right\}$ regions.   To estimate the Fourier image of the solution, we compute the exponential of the matrix $\Phi(\zeta) t$ thanks to the following lemma which avoids the use of eigenvectors and eigenspaces. 
\begin{lemma}[\cite{Putzer_1966}]\label{Theorem_Exponential_avoid_Jordan}
Assume that $A$ is an $n\times n$ matrix with the eigenvalues $\lambda_1,\lambda_2,...,\lambda_n$ (real or complex) written in some arbitrary but specified order and they are not necessary distinct. Then 
\begin{equation}\label{Formula_Exponential_No_Jordan}
e^{tA}=\sum_{j=0}^{n-1} r_{j+1}(t)P_j
\end{equation}
where 
\begin{equation*}
P_0=I,\qquad P_j=\prod_{k=1}^j
(A-\lambda_kI),\qquad j=1,...,n
\end{equation*}
and $r_1(t),...,r_n(t)$ are (real or complex) functions  defined by the following equations
\begin{equation}\label{Functions_r_j}
\left\{
\begin{array}{ll}
r^\prime_1(t)=\lambda_1r_1(t),&\qquad r_1(0)=1,\vspace{0.2cm}\\
r_2^\prime(t)=\lambda_2r_2(t)+r_1(t),&\qquad r_2(0)=0,\vspace{0.2cm}\\
\vdots\vspace{0.2cm}\\
r_n^\prime(t)=\lambda_nr_n(t)+r_{n-1}(t),&\qquad r_{n}(0)=0.
\end{array}
\right. 
\end{equation}
\end{lemma} 

Notice that Lemma \ref{Theorem_Exponential_avoid_Jordan} enables us to  compute the exponential of the matrix $\Phi(\zeta)t$ without the use of the canonical Jordan form which requires the knowledge of the eigenvectors  and the corresponding eigenspaces.

\subsection{Asymptotic expansion of the eigenvalues} 
In this subsection, we perform an asymptotic expansion of the eigenvalues of $\Phi(\zeta)$ in the low and high frequencies.   
\begin{proposition}\label{Propo_2}
Let $\gamma_1=0,\gamma_2>0$ and $\lambda_j(\zeta),\, 1\leq j\leq 6$ be the eigenvalues of $\Phi(\zeta)$. Assume that  
\begin{equation}\label{Main_Assump}
(k^2-1)l^2-1\neq 0. 
\end{equation}
Then,  as $|\xi|\rightarrow0$, we have  
\begin{equation}\label{Estimate_Eigenvalues_1}
Re(\lambda_j)(i\xi)=
\left\{
\begin{array}{ll}
-\dfrac{a^2l^2}{\gamma_2}|\xi|^2+O(|\xi|^3),&\qquad \text{for}\qquad j=1,\vspace{0.3cm}\\
-\hat{\beta}|\xi|^2+O(|\xi|^3),&\qquad \text{for}\qquad j=2,3,\vspace{0.3cm}\\
Re(\sigma_j)+O(|\xi|),&\qquad \text{for}\qquad j=4,5,6.
\end{array}
\right. 
\end{equation}

On the other hand,  as $|\xi|\rightarrow+\infty$, we obtain 

\begin{itemize}
\item For $a=k=1$, 
\begin{equation}\label{Estimate_Eigenvalues_infinity_1}
Re(\lambda_j)(i\xi)=
\left\{
\begin{array}{ll}
-\dfrac{l^2 \gamma_2}{2}|\xi|^{-2}+O(|\xi|^{-3}),&\qquad \text{for}\qquad j=1,2.
\vspace{0.2cm}\\
-\dfrac{\gamma_2}{6}+O(|\xi|^{-1}),&\qquad \text{for}\qquad j=3,4,\vspace{0.2cm}\\
-\dfrac{\gamma_2}{2}+O(|\xi|^{-1}),&\qquad \text{for}\qquad j=5,6.
\end{array}
\right. 
\end{equation}
\item For $a\neq 1$, we have 
\begin{equation}\label{Estimate_Eigenvalues_infinity_2}
Re(\lambda_j)(i\xi)=
\left\{
\begin{array}{ll}
-\dfrac{l^2\gamma_2}{2}|\xi|^{-2}+O(|\xi|^{-3}),&\qquad \text{for}\qquad j=1,2\vspace{0.2cm}\\
-\dfrac{l^2 \gamma_2}{2 (a-1)^2 (a+1)^2} |\xi|^{-4}+O(|\xi|^{-5}),&\qquad \text{for}\qquad j=3,4,\vspace{0.2cm}\\
-\dfrac{\gamma_2}{2}+O(|\xi|^{-1}),&\qquad \text{for}\qquad j=5,6.
\end{array}
\right. 
\end{equation}
\end{itemize}
\end{proposition}

\begin{proof}
\begin{itemize}
\item \textbf{Behavior of $\lambda_j(\zeta)$ when $|\zeta|\rightarrow 0$}
\end{itemize}
For $\gamma_1=0$, the characteristic polynomial of \eqref{Characteristic_Poly_2} takes the form 
 \begin{eqnarray}\label{Characteristic_Poly_3}
\det{\left(\lambda I-\Phi(\zeta)\right)}&=&\lambda^6+\gamma_2\lambda^5+\Big\{(k^2+1)(l^2-\zeta ^2) +1-a^2\zeta^2\Big\}\lambda^4\notag\\
&&+\gamma_2\Big\{ (k^2l^2+1)-(1+a^2)\zeta^2\Big\}\lambda^3\notag\\
&&+\left[(l^2-\zeta^2) \Big\{k^2(l^2-\zeta^2)+\Big(k^2-a^2(k^2+1)\zeta^2\Big)\Big\}\right]\lambda^2\\
&&+\gamma_2\Big\{ k^2 l^2  - a^2 k^2 l^2  \zeta^2  
+ a^2\zeta^4\Big\}\lambda-a^2k^2\zeta^2(l^2-\zeta^2)^2.\notag
\end{eqnarray}
For $|\zeta|\rightarrow 0$,  $\lambda_j(\zeta)$ has the  asymptotic expansion as in \eqref{Asymptotic_Expansion_2}.
Consequently,  we have from the above equation that 
\begin{equation*}
\left\{
\begin{array}{ll}
\lambda_j^{(0)}=0,\qquad \lambda_j^{(1)}=0,\qquad \lambda_j^{(2)}=\dfrac{a^2 l^2}{\gamma_2},&\qquad \text{for}\qquad j=1,\vspace{0.2cm}\\
\lambda_j^{(0)}=\pm ikl,\qquad \lambda_j^{(1)}=0,\qquad \lambda_j^{(2)}=\hat{\beta}\pm i\hat{\delta}&\qquad \text{for}\qquad j=2,3,\vspace{0.2cm}\\
\lambda_j^{(0)}=\sigma_j,&\qquad \text{for}\qquad j=4,5,6,
\end{array}
\right. 
\end{equation*}
with 
\begin{equation*}
\hat{\beta}=\frac{\gamma _2 k^2 \left(\left(k^2-1\right) l^2-1\right)^2}{2 \left(\gamma _2^2 \left(k^2 l^2-1\right)^2+k^2 l^2 \left(\left(k^2-1\right) l^2-1\right)^2\right)}. 
\end{equation*}
and 
\begin{equation*}
\hat{\delta}=-\frac{k l \left(\gamma _2^2 \left(k^2 l^2-1\right)+\left(k-k \left(k^2-1\right) l^2\right)^2\right)}{2 \left(\gamma _2^2 \left(k^2 l^2-1\right)^2+k^2 l^2 \left(\left(k^2-1\right) l^2-1\right)^2\right)}. 
\end{equation*}
It is clear that under the assumption \eqref{Main_Assump}, $\hat{\beta}>0$.

On the other hand
 $\sigma_j,\, 4\leq j\leq 6$ are the solutions $Z$ of the cubic equation 
 \begin{equation}\label{Cubic_Equation}
Z^3+\gamma_2Z^2+(l^2+1)Z+\gamma_2=0.
\end{equation}

We can easily show as before, that  $Re(\sigma_j)<0.$
Now, we want to see the multiplicity of the roots of \eqref{Cubic_Equation}. As we have seen before, equation \eqref{Cubic_Equation} has at least one negative real root. In order to know the nature of the other two roots of the equation (in general)  
\begin{equation*}
ax^3+bx^2+cx+d=0
\end{equation*}
we use the Cardano method and investigate the sign of  
\begin{equation*}
D=Q^3+R^2,
\end{equation*}
with 
\begin{equation*}
Q=\frac{3ac-b^2}{9a^2}\qquad \text{and}\qquad R=\frac{9abc-2b^3-27a^2d}{54a^3}.
\end{equation*}
Now, for  equation \eqref{Cubic_Equation}, we have 
\begin{equation*}
Q=\frac{3(l^2+1)-\gamma_2^2}{9},\qquad R=\frac{9\gamma_2(l^2+1)-2\gamma_2^3-27\gamma_2}{54}.
\end{equation*}
We define 
\begin{equation*}
D=Q^3+R^2. 
\end{equation*}
Thus, we have the following cases 
\begin{itemize}
\item If $D>0$, then \eqref{Cubic_Equation} has one real and two complex conjugate roots.
\item If $D<0$, there are three  distinct real roots.
\item  If $D=0$, there is one real root and another real root of double multiplicity. 
\end{itemize}
To do so, we compute:  
\begin{equation*}
D=\frac{1}{108} \left(4 \gamma _2^4- \left(l^4+20 l^2-8\right)\gamma _2^2+4 \left(l^2+1\right)^3\right). 
\end{equation*}
In order to determine the sign of $D$, we consider the quadratic polynomial 
\begin{equation*}
\Lambda =4\omega^2-\left(l^4+20 l^2-8\right)\omega+4 \left(l^2+1\right)^3,\qquad \omega =\gamma _2^2. 
\end{equation*}

Now, it is clear that if $l^2<8$, then the discriminant of $\Lambda$ is:
\begin{equation*}
\Delta =\left(l^4+20 l^2-8\right)^2-64\left(l^2+1\right)^3 =l^2 \left(l^2-8\right)^3<0. 
\end{equation*}
Therefore, $\Lambda>0$, hence $D>0$. 
Consequently,  \eqref{Cubic_Equation} has one real root and two complex conjugate roots. In this case 
\begin{equation}\label{Eigenvalues_0}
\lambda_4=\sigma_4,\quad \lambda_{5}=\sigma_5+i\hat{\sigma}_5,\quad  \lambda_{6}=\sigma_5-i\hat{\sigma}_5, 
\end{equation}
with $\sigma_i<0,\, i=4,5$. 

If $l^2> 8$, then we consider  the equation 
\begin{equation*}
4 \gamma _2^4+4 \left(l^2+1\right)^3-\gamma _2^2 \left(l^4+20 l^2-8\right)=0,
\end{equation*}
written as 
\begin{equation*}
\gamma _2^2=\frac{1}{8} \left(-8 + 20 l^2 + l^4 \pm l \sqrt{(-8 + l^2)^3}\right)>0. 
\end{equation*}
We put 
\begin{equation}\label{gamma_1_gamma_2}
\hat{\gamma}_1=\frac{1}{8} \left(-8 + 20 l^2 + l^4 + l \sqrt{(-8 + l^2)^3}\right),\quad \hat{\gamma}_2=\frac{1}{8} \left(-8 + 20 l^2 + l^4 - l \sqrt{(-8 + l^2)^3}\right).
\end{equation}
Then, for $\gamma_2^2\in (0,\hat{\gamma_1})\cup (\hat{\gamma_2},\infty)$, then $D>0$ and then we have the same situation as in \eqref{Eigenvalues_0}.  If $\gamma_2^2\in(\hat{\gamma}_1,\hat{\gamma}_2),$ then \eqref{Cubic_Equation} has three real roots $\sigma_4\neq \sigma_5\neq\sigma_6 .$

If $l^2=8$, then $D=0$, and in this case  \eqref{Cubic_Equation} has three real roots $\sigma_4\neq \sigma_5=\sigma_6 .$

Consequently, for $|\xi|\rightarrow 0$, we have 
\begin{itemize}
\item for $l^2<8$ or ($l^2>8$ and $\gamma_2^2\in (0,\hat{\gamma_1})\cup (\hat{\gamma_2},\infty)$)
\begin{equation}\label{Estimate_Eigenvalues_0}
\lambda_j(i\xi)=
\left\{
\begin{array}{ll}
-\dfrac{a^2l^2}{\gamma_2}\xi^2+O(|\xi|^3),&\qquad \text{for}\qquad j=1,\vspace{0.3cm}\\
-\hat{\beta}\xi^2\pm i(lk- \hat{\delta}\xi^2)+O(|\xi|^3),&\qquad \text{for}\qquad j=2,3,\vspace{0.3cm}\\
\sigma_4+O(|\xi|),&\qquad \text{for}\qquad j=4,\vspace{0.3cm}\\
\sigma_5\pm i\hat{\sigma}_5+O(|\xi|),&\qquad \text{for}\qquad j=5,6.
\end{array}
\right. 
\end{equation}

\item For $l^2>8$ and  $\gamma_2^2\in(\hat{\gamma}_1,\hat{\gamma}_2).$  
In this case, we have  the following expansion for $|\xi|\rightarrow 0$, 
\begin{equation}\label{Estimate_0_1}
\lambda_j(i\xi)=
\left\{
\begin{array}{ll}
-\sigma_0\xi^2+O(|\xi|^3),&\qquad \text{for}\qquad j=1,\vspace{0.3cm}\\
-\hat{\beta}\xi^2\pm i(lk- \hat{\delta}\xi^2)+O(|\xi|^3),&\qquad \text{for}\qquad j=2,3,\vspace{0.3cm}\\
\sigma_i+O(|\xi|),&\qquad \text{for}\qquad j=4,5,6.
\end{array}
\right. 
\end{equation}
\item For $l^2=8$, we have  
\begin{equation*}
\lambda_5(i\xi)=\lambda_6(i\xi)=\sigma_5 
\end{equation*}
and $\lambda_j(i\xi),\, j=1,2,3,4$ are the same as before. 
\item\textbf{Behavior of $\lambda_j(\zeta)$ when $|\zeta|\rightarrow \infty$.}
\end{itemize}

Set $\lambda_j(\zeta)=\zeta\mu_j(\zeta^{-1})$.
Hence, for $|\zeta|\rightarrow \infty$ then \eqref{Charact_Equation_mu} takes the form 
\begin{eqnarray}\label{Charact_Equation_mu_2}
\Psi(\zeta^{-1}) &=&\zeta^6\det {(\mu I+(A+\zeta^{-1} L))}\notag\vspace{0.3cm}\\
&=& \mu^6+\gamma_2\zeta^{-1}\mu^5+\Big\{(k^2+1)(l^2\zeta^{-2}-1) +\zeta^{-2}-a^2\Big\}\mu^4\notag\\
&&+\gamma_2\Big\{(k^2l^2+1)\zeta^{-2}-(1+a^2)\Big\}\zeta^{-1}\mu^3\notag\\
&&+\left[(l^2\zeta^{-2}-1) \Big\{k^2(l^2\zeta^{-2}-1)+\Big(k^2\zeta^{-2}-a^2(k^2+1)\Big)\Big\}\right]\mu^2\\
&&+\gamma_2\Big\{ k^2 l^2 \zeta^{-4}  - a^2 k^2 l^2 \zeta^{-2}  
+ a^2  \Big\}\zeta^{-1}\mu-a^2k^2(l^2\zeta^{-2}-1)^2=0.\notag
\end{eqnarray}
We assume that $\mu_j(\zeta^{-1})$, have the asymptotic expansion \eqref{Asymptotic_Expansion_mu}, then we get by a direct computation, using \eqref{Charact_Equation_mu_2}:

For $a=k=1$
\begin{equation*}
\left\{
\begin{array}{ll}
\mu_j^{(0)}=\pm 1,\qquad \mu_j^{(1)}=0,\qquad \mu_j^{(2)}=\mp \dfrac{l^2}{2},\qquad \mu_j^{(3)}=\dfrac{l^2 \gamma_2}{4},
&\qquad  \text{for}\qquad j=1,2\vspace{0.2cm}\\
\mu_j^{(0)}=\pm a,\qquad \mu_j^{(1)}=-\dfrac{\gamma_2}{6}\qquad\qquad \qquad &\qquad \text{for}\qquad j=3,4,\vspace{0.2cm}\\
\mu_j^{(0)}=\pm k \qquad \mu_j^{(1)}=-\dfrac{\gamma_2}{2}&\qquad \text{for}\qquad j=5,6.
\end{array}
\right. 
\end{equation*}

 For $a\neq 1$, we have 

\begin{equation*}
\left\{
\begin{array}{ll}
 \mu_j^{(2)}=\pm\dfrac{l^2(1-a^2)+1}{2 \left(a^2-1\right)},\qquad \mu_j^{(3)}=\dfrac{l^2 \gamma_2}{2},&\qquad  \text{for}\qquad j=1,2\vspace{0.3cm}\\
 \mu_j^{(1)}=0,\qquad \mu_j^{(2)}=\pm\dfrac{a}{2 \left(1-a^2\right)},\qquad \mu_j^{(3)}=0,&\qquad  \text{for}\qquad j=3,4,\vspace{0.3cm}\\
 \mu_j^{(4)}=\pm\dfrac{a \left(a^2 \left(4 l^2-1\right)-4 l^2-3\right)}{8 \left(a^2-1\right)^3},\qquad \mu_j^{(5)}=-\dfrac{l^2 \gamma_2}{2 (a-1)^2 (a+1)^2}&\qquad  \text{for}\qquad j=3,4,\vspace{0.3cm}\\
 \mu_j^{(0)}=\pm k \qquad \mu_j^{(1)}=-\dfrac{\gamma_2}{2}&\qquad  \text{for}\qquad j=5,6.
\end{array}
\right. 
\end{equation*}
Consequently, we deduce \eqref{Estimate_Eigenvalues_infinity_1} and \eqref{Estimate_Eigenvalues_infinity_2}.
\end{proof}

\subsection{The estimates in the low frequency region $\Upsilon_L$}
For $\xi\in \Upsilon_L$, we have the following estimates.  
\begin{proposition}\label{Proposition_Low}
There exists two positive constants $\hat{c}_1$ and $\hat{c}_2$ such that the solution $\hat{U}(\xi,t)$ of \eqref{Fourier_system} satisfies for $t$ large enough in $\Upsilon_{L}$ the following estimate:
\begin{equation}\label{Low_F_Estimate}
\vert \hat{U}( \xi ,t)\vert \leq \hat{c}_1 e^{-\hat{c}_2 |\xi|^2 t}\vert \hat{U}( \xi ,0)\vert,\qquad \forall t\geq 0.
\end{equation}
\end{proposition}
\begin{proof}
We discuss the following cases 
\begin{description}
\item[ Case 1] $l^2<8$ or  ($l^2>8$ and $\gamma_2^2\in (0,\hat{\gamma_1})\cup (\hat{\gamma_2},\infty)$) where $\hat{\gamma_1}$ and $\hat{\gamma_2}$ are defined in \eqref{gamma_1_gamma_2}. 
\end{description}

In this case, we have  (by neglecting  the small terms) that in $\Upsilon_L,$ the eigenvalues are:
\begin{equation}\label{Estimate_Eigenvalues_0_1}
\lambda_j(i\xi)=
\left\{
\begin{array}{ll}
-\sigma_0\xi^2+O(|\xi|^3),&\qquad \text{for}\qquad j=1,\vspace{0.3cm}\\
-\hat{\beta}\xi^2\pm i(lk- \hat{\delta}\xi^2)+O(|\xi|^3),&\qquad \text{for}\qquad j=2,3,\vspace{0.3cm}\\
\sigma_4+O(|\xi|),&\qquad \text{for}\qquad j=4,\vspace{0.3cm}\\
\sigma_5\pm i\hat{\sigma}_5+O(|\xi|),&\qquad \text{for}\qquad j=5,6.
\end{array}
\right. 
\end{equation}
where $\sigma_0=\frac{a^2l^2}{\gamma_2}$. 
Using the above eigenvalues, we may find the functions $r_j(t),\, 1\leq j\leq 6$ as solutions of \eqref{Functions_r_j}. Indeed, we have 
\begin{equation*}
\left\{
\begin{array}{ll}
r_1(t)=e^{\lambda_1(i\xi )t},&\vspace{0.2cm}\\
r_j(t)=\dint_0^te^{\lambda_j(t-s)}r_{j-1}(s)ds&\qquad \text{for}\qquad 2\leq j\leq 6.
\end{array}
\right. 
\end{equation*}
Consequently, since all the eigenvalues are of multiplicity one we deduce for $\xi\neq 0$
\begin{eqnarray}\label{r_j_formula}
r_j(t)=\sum_{i=1}^{j}\Big(\prod_{k=1,(k\neq i)}^j \frac{e^{\lambda_it}}{(\lambda_i-\lambda_k)}\Big),\qquad 1\leq j\leq 6.
\end{eqnarray}
For instance, we have 
\begin{equation*}
\left\{
\begin{array}{ll}
r_1(t)&=e^{\lambda_1t}\vspace{0.2cm}\\
r_2(t)&=\dfrac{e^{\lambda_2 t}-e^{\lambda_1t}}{\lambda_2-\lambda_1}
\vspace{0.3cm}\\
r_3(t)&=\dfrac{e^{\lambda _1 t}}{\left(\lambda _1-\lambda _2\right) \left(\lambda _1-\lambda _3\right)}+\dfrac{e^{\lambda _2 t}}{\left(\lambda _2-\lambda _1\right) \left(\lambda _2-\lambda _3\right)}+\dfrac{e^{\lambda_3 t}}{\left(\lambda _1-\lambda _3\right) \left(\lambda _2-\lambda _3\right)}, 
\end{array}
\right. 
\end{equation*} 

Hence, using \eqref{Formula_Exponential_No_Jordan}, we deduce that 
\begin{eqnarray}\label{Formula_Exponential}
\hat{U}(\xi,t)&=&e^{\Phi(i\xi)t}\hat{U}_0(\xi)=\sum_{j=0}^{5}r_{j+1}(t)P_j\hat{U}_0(\xi).
\end{eqnarray}
By doing some tedious computations, we may show that 
 for $|\xi|\rightarrow 0$
 \begin{equation}\label{Estimate_Product}
\left|\prod_{k=1,(k\neq i)}^j \frac{1}{(\lambda_i-\lambda_k)}\right|\leq \Lambda,
\end{equation}
where $\Lambda$ is a positive constant and $1\leq j\leq 6$. For example, we have 
\begin{eqnarray*}
&&\left|\frac{1}{\lambda _2-\lambda _1}\right|\leq \kappa_1\\
&&\left|\frac{1}{\left(\lambda _1-\lambda _2\right) \left(\lambda _1-\lambda _3\right)}\right| \leq \kappa_2,\\
&&\left|\frac{1}{\left(\lambda _2-\lambda _1\right) \left(\lambda _2-\lambda _3\right)}\right| \leq \kappa_3
\end{eqnarray*}
where $\kappa_i,\ 1\leq i\leq 3$ are positive constants. 
From \eqref{r_j_formula} and \eqref{Estimate_Product}, we deduce 
\begin{equation}\label{Estimates_r_j}
 |r_j(t)| \leq c_1e^{-c|\xi|^2t},\quad j\in\{1,2,3,4,5,6\}.
\end{equation}
On the other hand, since $|\xi|$ is close to zero we have
\begin{equation}\label{P_j_Estimate}
|P_j|\leq C, \qquad j\in\{1,2,3,4,5,6\}. 
\end{equation}
Consequently, we obtain  from \eqref{Formula_Exponential} that for $|\xi|\rightarrow 0$
\begin{equation*}
|e^{\Phi(i\xi)t}|\leq Ce^{-|\xi|^2 t},
\end{equation*}
which leads to \eqref{Low_F_Estimate}.

\begin{description}
\item[ Case 2]  $l^2>8$ and $\gamma_2^2\in(\hat{\gamma}_1,\hat{\gamma}_2).$  
\end{description}
In this case, the eigenvalues of $\Phi(\xi)$ has the following expansion for $|\xi|\rightarrow 0$, 
\begin{equation}\label{Estimate_Eigenvalues_0_2}
\lambda_j(i\xi)=
\left\{
\begin{array}{ll}
-\sigma_0\xi^2+O(|\xi|^3),&\qquad \text{for}\qquad j=1,\vspace{0.3cm}\\
-\hat{\beta}\xi^2\pm i(lk- \hat{\delta}\xi^2)+O(|\xi|^3),&\qquad \text{for}\qquad j=2,3,\vspace{0.3cm}\\
\sigma_i+O(|\xi|),\quad Re(\sigma_i)<0,&\qquad \text{for}\qquad j=4,5,6.
\end{array}
\right. 
\end{equation}
One can easily show, as previously,  that 
\begin{equation*}
|e^{\Phi(i\xi)t}|\leq Ce^{-|\xi|^2 t}.
\end{equation*}

\begin{description}
\item[Case 3]  $l^2=8$.
\end{description}
In this case, the main difference with the other cases is the presence of an eigenvalue with multiplicity 2, namely:
\begin{equation*}
\lambda_5(i\xi)=\lambda_6(i\xi)=\sigma_5,\quad  \mbox{ for } \xi\in\Upsilon_L.
\end{equation*}

Consequently, $r_6(t)$ in the case $l^2=8$ becomes
\begin{eqnarray}\label{r_6_formula}
r_6(t)&=&\sum_{i=1}^{4}\Big(\prod_{k=1,(k\neq i)}^4 \frac{e^{\lambda_it}}{(\lambda_i-\lambda_k)(\lambda_i-\lambda_5)^2}\Big)\notag\\
&&-e^{\lambda_5 t}\sum_{i=1}^{4}\Big(\prod_{k=1,(k\neq i)}^4 \frac{1}{(\lambda_i-\lambda_k)(\lambda_i-\lambda_5)^2}\Big)\notag\\
&&+te^{{\lambda_5 t}}\Big(\prod_{k=1}^4  \frac{1}{(\lambda_i-\lambda_5)}\Big).
\end{eqnarray}
We can prove as before that 
\begin{equation*}
 |r_6(t)| \leq Ce^{-c|\xi|^2t}+Cte^{\sigma_5t}.
\end{equation*}
Hence, for $t$ large enough there exist $C>0$ and $c_0>0$ such that
$$te^{\sigma_5t}\leq Ce^{-c_0t}.$$
Consequently, we have from \eqref{Formula_Exponential} that for $|\xi|\rightarrow 0$ and $t$ large enough
\begin{equation*}
|e^{\Phi(i\xi)t}|\leq Ce^{-\tilde{c}|\xi|^2 t},
\end{equation*}
for some $\tilde{c}>0$.  This leads to \eqref{Low_F_Estimate} and ends the proof of Proposition \ref{Proposition_Low}. 
\end{proof}
\subsection{The estimates in the high  frequency region $\Upsilon_H$}
For the high frequency region, we have the following estimates.
\begin{proposition}\label{Proposition_High}
Assume that $\left(k^2-1\right) l^2-1\neq 0$. Then, 
there exists two positive constants $\hat{c}_3$ and $\hat{c}_4$ such that the solution $\hat{U}(\xi,t)$ of \eqref{Fourier_system} satisfies in $\Upsilon_{H}$ the estimates:
\begin{itemize}
\item If $a=k=1$, then
\begin{equation}\label{High_F_Estimate_1}
\vert \hat{U}( \xi ,t)\vert \leq \hat{c}_3 |\xi|^2e^{-\hat{c}_4 |\xi|^{-2} t}\vert \hat{U}( \xi ,0)\vert,\qquad \forall t\geq 0.
\end{equation}
\item If $a\neq 1$, then we have 
\begin{equation}\label{High_F_Estimate_2}
\vert \hat{U}( \xi ,t)\vert \leq \hat{c}_3 |\xi|^6e^{-\hat{c}_4 |\xi|^{-2} t}\vert \hat{U}( \xi ,0)\vert,\qquad \forall t\geq 0.
\end{equation}
\end{itemize}

\end{proposition} 
\begin{proof}
Now, for $|\xi|\rightarrow \infty$, we have the following expansion of the eigenvalues:

For $a=k=1$, we get  
\begin{equation}
\lambda_j(i\xi)=
\left\{
\begin{array}{ll}
\pm i\xi\pm i\dfrac{l^2}{2}\xi^{-1}-\dfrac{l^2 \gamma_2}{2}\xi^{-2}+O(|\xi|^{-3}),&\qquad \text{for}\qquad j=1,2.
\vspace{0.2cm}\\
\pm i\xi-\dfrac{\gamma_2}{6}+O(|\xi|^{-1}),&\qquad \text{for}\qquad j=3,4,\vspace{0.2cm}\\
\pm i\xi-\dfrac{\gamma_2}{2}+O(|\xi|^{-1}),&\qquad \text{for}\qquad j=5,6.
\end{array}
\right. 
\end{equation}
Using the same method as before, we have for $|\xi|\rightarrow \infty$
\begin{eqnarray}\label{Formula_Exponential_2}
\hat{U}(\xi,t)&=&e^{\Phi(i\xi)t}\hat{U}_0(\xi)=\sum_{j=0}^{5}r_{j+1}(t)P_j\hat{U}_0(\xi),
\end{eqnarray}
where $r_j(t)$ are given by \eqref{r_j_formula} and $P_j$ are defined as before. 

First of all, it is straightforward to see that for $|\xi|\rightarrow \infty$ and for all $1\leq j\leq 6$, we have
\begin{equation}\label{Estimate_lambda_j}
|e^{\lambda_j(\xi) t}|\leq Ce^{-c\xi^{-2} t}.
\end{equation}
 On the other hand, we have 
 \begin{equation}\label{Estimate_Eigenvalues_2}
\left\{
\begin{array}{ll}
\left|\dfrac{1}{\lambda _2-\lambda _1}\right|\leq C|\xi|^{-1},\vspace{0.3cm}\\
\left|\sum_{i=1}^{3}\Big(\prod_{k=1,(k\neq i)}^3 \dfrac{1}{(\lambda_i-\lambda_k)}\Big)\right|\leq C|\xi|^{-1}\vspace{0.3cm}\\
\left|\sum_{i=1}^{4}\Big(\prod_{k=1,(k\neq i)}^4 \dfrac{1}{(\lambda_i-\lambda_k)}\Big)\right|\leq C|\xi|^{-2}\vspace{0.3cm}\\
\left|\sum_{i=1}^{5}\Big(\prod_{k=1,(k\neq i)}^5 \dfrac{1}{(\lambda_i-\lambda_k)}\Big)\right|\leq C|\xi|^{-2}\vspace{0.3cm}\\
\left|\sum_{i=1}^{6}\Big(\prod_{k=1,(k\neq i)}^6 \dfrac{1}{(\lambda_i-\lambda_k)}\Big)\right|\leq C|\xi|^{-3}. 
\end{array}
\right.
\end{equation}
For the matrices $P_j,\, 0\leq j\leq 5$, we have as before 
\begin{equation*}
|P_0|=|I|\leq C, 
\end{equation*}
and since 
\begin{equation*}
|\Phi(i\xi)|=|-(L+i\xi A)|\leq C(1+|\xi|),
\end{equation*}
then, we have 
\begin{eqnarray*}
|P_1|=|\Phi(i\xi)-\lambda_1I|&=&|-(L+i\xi A)-\lambda_1I|\\
&\leq&C|\xi|, 
\end{eqnarray*}
for $|\xi|\rightarrow\infty$.  Also we have that
\begin{eqnarray*}
|P_2|&=&|P_1.(\Phi(i\xi)-\lambda_2I)|\\
&\leq&C |P_1|.(|\Phi(i\xi)|+|\lambda_2|)\\
&\leq& C|\xi|^2.
\end{eqnarray*}
Hence, we can deduce easily in the same fashion that for  $|\xi|\rightarrow\infty$
\begin{equation*}
|P_j|\leq C |\xi|^j \quad \mbox{ for all } j\in\{0,1,2,3,4,5\}.
\end{equation*}
Consequently, since all the eigenvalues are of multiplicity one when $|\xi|$ is large and by using \eqref{Formula_Exponential_2}, \eqref{Estimate_lambda_j}, \eqref{Estimate_Eigenvalues_2} and the above estimates, we obtain 
\begin{eqnarray*}
|\hat{U}(\xi,t)|&\leq&\sum_{j=0}^{5} |r_{j+1}(t)||P_j|\vert U_0(\xi)\vert\\
&\leq&C|\xi|^2e^{-c|\xi|^{-2} t}\vert U_0(\xi)\vert,
\end{eqnarray*}
for $|\xi|\rightarrow \infty.$

For $a\neq 1$, we obtain  
\begin{equation}
\lambda_j(i\xi)=
\left\{
\begin{array}{ll}
\pm i\dfrac{l^2(1-a^2)+1}{2 \left(a^2-1\right)}\xi^{-1}-\dfrac{l^2 \gamma_2}{2}\xi^{-2}+O(|\xi|^{-3}),&\qquad \text{for}\qquad j=1,2.
\vspace{0.2cm}\\
\mp i\dfrac{a}{2 \left(1-a^2\right)}\xi^{-1}\pm i\dfrac{a \left(a^2 \left(4 l^2-1\right)-4 l^2-3\right)}{8 \left(a^2-1\right)^3}\xi^{-3}\\
\qquad\qquad-\dfrac{l^2 \gamma_2}{2 (a-1)^2 (a+1)^2}\xi^{-4}+O(|\xi|^{-5}),&\qquad \text{for}\qquad j=3,4,\vspace{0.2cm}\\
\pm ik\xi-\dfrac{\gamma_2}{2}+O(|\xi|^{-1}),&\qquad \text{for}\qquad j=5,6.
\end{array}
\right. 
\end{equation}
In this case, we have 
 \begin{equation}\label{Estimate_Eigenvalues_3}
\left\{
\begin{array}{ll}
\left|\dfrac{1}{\lambda _2-\lambda _1}\right|\leq C|\xi|,\vspace{0.3cm}\\
\left|\sum_{i=1}^{3}\Big(\prod_{k=1,(k\neq i)}^3 \dfrac{1}{(\lambda_i-\lambda_k)}\Big)\right|\leq C|\xi|^{2}\vspace{0.3cm}\\
\left|\sum_{i=1}^{4}\Big(\prod_{k=1,(k\neq i)}^4 \dfrac{1}{(\lambda_i-\lambda_k)}\Big)\right|\leq C|\xi|^{3}\vspace{0.3cm}\\
\left|\sum_{i=1}^{5}\Big(\prod_{k=1,(k\neq i)}^5 \dfrac{1}{(\lambda_i-\lambda_k)}\Big)\right|\leq C|\xi|^{2}\vspace{0.3cm}\\
\left|\sum_{i=1}^{6}\Big(\prod_{k=1,(k\neq i)}^6 \dfrac{1}{(\lambda_i-\lambda_k)}\Big)\right|\leq C|\xi|. 
\end{array}
\right.
\end{equation}
Also, in this case, we have for the matrices $P_j,\,0\leq j\leq 5$,  as before  
\begin{equation*}
|P_j|\leq |\xi|^j,\qquad 1\leq j\leq 5.
\end{equation*}
Consequently, since 
\begin{equation*}
|e^{\lambda_j(\xi) t}|\leq Ce^{-c\xi^{-4} t},\qquad 1\leq j\leq 6,
\end{equation*}
then, we have by the same method as above 
\begin{eqnarray*}
|\hat{U}(\xi,t)|&\leq&\sum_{j=0}^{5} |r_{j+1}(t)||P_j|\vert U_0(\xi)\vert\\
&\leq&C|\xi|^6e^{-c|\xi|^{-4} t}\vert U_0(\xi)\vert,
\end{eqnarray*}
for $|\xi|\rightarrow \infty.$ This proves \eqref{High_F_Estimate_2}, and concludes the proof of Proposition \ref{Proposition_High}. 
\end{proof}
\subsection{The estimates in the middle   frequency region $\Upsilon_M$}
For the middle frequency region, we need to show first that for $\xi\in \Upsilon_M$, the matrix $\Phi(i\xi)$ has no pure imaginary eigenvalue. We have the following lemma. 
\begin{lemma}\label{Lemma_Reissig}
The matrix $\Phi(i\xi)=-L-i\xi A$ has no pure imaginary eigenvalue in $\Upsilon_M$. 
\end{lemma}
\begin{proof}
We consider  the characteristic equation  \eqref{Characteristic_Poly_3}.
We argue by contradiction. Assume that there exists an eigenvalue $\lambda_0(\xi)=i\alpha$ a solution of \eqref{Characteristic_Poly_3} with $\alpha\in \R$. We plugg $\lambda_0$ into \eqref{Characteristic_Poly_3} and  splitt the real and imaginary parts.
Hence,

the real part: 
\begin{eqnarray}\label{Re_Equation}
&&\alpha ^6- \left(\xi ^2 \left(a^2+k^2+1\right)+\left(k^2+1\right) l^2+1\right)\alpha ^4\notag\\
&&+ \left(l^2+\xi ^2\right) \left\{k^2 \left(\left(a^2+1\right) \xi ^2+l^2+1\right)+a^2 \xi ^2\right\}\alpha ^2-a^2 k^2 \xi ^2 \left(l^2+\xi ^2\right)^2=0.
\end{eqnarray}

The imaginary part:
\begin{eqnarray}\label{Im_Equation}
\alpha\left[\alpha ^4 - \left( \left(a^2+1\right) \xi ^2+k^2 l^2+1\right)\alpha ^2+  \left(k^2 l^2 \left(a^2 \xi ^2+1\right)+a^2 \xi ^4\right)\right]=0.
\end{eqnarray}
On one hand,  if $\alpha=0$ in equation \eqref{Im_Equation}, then, we have from \eqref{Re_Equation}
\begin{equation*}
a^2 k^2 \xi ^2 \left(l^2+\xi ^2\right)^2=0,
\end{equation*}
which is a contradiction since $\xi\in \Upsilon_M$. On the other hand, we get  from \eqref{Im_Equation} that 
\begin{equation*}
S:=\alpha ^4 - \left( \left(a^2+1\right) \xi ^2+k^2 l^2+1\right)\alpha ^2+  \left(k^2 l^2 \left(a^2 \xi ^2+1\right)+a^2 \xi ^4\right)=0.
\end{equation*}
On the other hand, equation \eqref{Re_Equation} can be rewritten as 
\begin{eqnarray*}
\alpha^2 \left[S-(\xi^2k^2+l^2)\alpha^2+\left(l^2+\xi ^2\right) (k^2\xi^2+k^2l^2)+\xi^2(k^2+l^2a^2)\right]-a^2 k^2 \xi ^2 \left(l^2+\xi ^2\right)^2=0
\end{eqnarray*}
Thus, we obtain 
\begin{equation*}
(\xi^2k^2+l^2)\alpha^4-\left\{\left(l^2+\xi ^2\right) (k^2\xi^2+k^2l^2)+\xi^2(k^2+l^2a^2)\right\}\alpha^2+a^2 k^2 \xi ^2 \left(l^2+\xi ^2\right)^2=0. 
\end{equation*}
Solving the above equation, we obtain 
\begin{equation*}
\alpha^2_{1,2}=\frac{a^2 l^2 \xi ^2+k^2 \left(\left(l^2+\xi ^2\right)^2+\xi ^2\right)}{2 \left(k^2 \xi ^2+l^2\right)}\pm\sqrt{\frac{a^2 k^2 \xi ^2 \left(l^2+\xi ^2\right)^2}{k^2 \xi ^2+l^2}+\frac{\left(a^2 l^2 \xi ^2+k^2 \left(\left(l^2+\xi ^2\right)^2+\xi ^2\right)\right)^2}{4 (k^2 \xi ^2+l^2)^2}}
\end{equation*}
This leads to $\alpha_2^2<0,$ which is a contradiction. 
\end{proof}
\begin{lemma}\label{Lemma_Eigenvalue_Middle_Frequency}
There is a constant $C>0$, such that 
\begin{equation}\label{Bound_M_Eigenvalues}
Re(\lambda_j(i\xi))<-C<0, 
\end{equation}
for all $\xi\in \Upsilon_M$, where $\lambda_j(i\xi),\, 1\leq j\leq 6$ are the eigenvalues of the matrix $\Phi(i\xi)$. 
\end{lemma}
\begin{proof}
Since by \cite[p.63]{Kat76_2} the eigenvalues are analytic in $\zeta=i\xi$, it follows by continuity and Lemma \ref{Lemma_Reissig} that if the real part of the eigenvalues is negative on the boundary and it cannot be zero inside the domain (by Lemma \ref{Lemma_Reissig}) then the real part of the eigenvalues must be negative inside the domain also. Which concludes the proof of  Lemma \ref{Lemma_Eigenvalue_Middle_Frequency}. 
\end{proof}

\begin{proposition}\label{Proposition_Middle}
Assume that $\Phi(i\xi) $ has an eigenvalue of multiplicity $m$. Then,  there exist two positive constant $\hat{c}_5$ and $C$ such that the solution $\hat{U}(\xi,t)$ of \eqref{Fourier_system} satisfies in $\Upsilon_{M}$ the estimates:
\begin{equation}\label{Middle_F_Estimate}
\vert \hat{U}( \xi ,t)\vert \leq \hat{c}_5(1+|\xi|^{2m}t^m) e^{-C t}\vert \hat{U}( \xi ,0)\vert,\qquad \forall t\geq 0.
\end{equation}
\end{proposition}
\begin{proof}
First,  if the eigenvalues of $\Phi(i\xi)$ are simple then, it is not hard to see that for $\xi\in \Upsilon_M$, we have 
\begin{eqnarray}\label{r_j_formula_M}
\left|\sum_{i=1}^{j}\Big(\prod_{k=1,(k\neq i)}^j \frac{1}{(\lambda_i-\lambda_k)}\Big)\right|\leq C,\qquad 1\leq j\leq 6.
\end{eqnarray}
Also, 
\begin{equation}\label{P_j_M}
|P_j|\leq C.
\end{equation}
Thus, using \eqref{Bound_M_Eigenvalues} together with \eqref{r_j_formula_M} and \eqref{P_j_M}, we deduce that \eqref{Middle_F_Estimate} holds true. 

Second, let us discuss, for example, the case of one double roots.  Assume that there exists $\xi_0\in\Upsilon_M$ such that 
$\lambda_5(i\xi_0)=\lambda_6(i\xi_0)$ (that is $m=1$ in \eqref{Middle_F_Estimate}). Hence, as we have seen in the proof of Proposition \ref{Proposition_Low} (case 3), 
and by \eqref{Bound_M_Eigenvalues}
\begin{equation*}
|r_j(t)|\leq C_1e^{-Ct},\qquad 1\leq j\leq 5
\end{equation*}
and 
\begin{equation*}
|r_6(t)|\leq C_1(1+t)e^{-Ct}. 
\end{equation*}
On the other hand, we can show that 
\begin{equation*}
|P_j(i\xi_0)|\leq C,\qquad 0\leq j\leq 5.  
\end{equation*}
In particular 
\begin{equation*}
|P_5(i\xi_0)|\leq C\leq C|\xi_0|^2. 
\end{equation*}
Therefore, by collecting the above estimates, we get 
\begin{equation}\label{Middle_F_Estimate}
\vert \hat{U}( \xi_0 ,t)\vert \leq (C_3 +|\xi_0|^2 t)e^{-C t}\vert \hat{U}( \xi_0 ,0)\vert,\qquad \forall t\geq 0.
\end{equation}
Now, let us consider, for $\xi$ varying in small neighborhood of $\xi_0$, the problem 
\begin{equation}\label{Solution_Equation}
\hat{U}_t-\Phi(i\xi_0)\hat{U}=\left(\Phi(i\xi)-\Phi(i\xi_0)\right)\hat{U}.
\end{equation}
Let $M(t,s,\xi_0)$ be the fundamental matrix of the operator $\partial_t-\Phi(i\xi_0)$. Then, it does satisfy the system 
\begin{equation*}
\left\{
\begin{array}{ll}
\partial_tM(t,s,\xi_0)-\Phi(i\xi_0)M(t,s,\xi_0)=0,\vspace{0.2cm}\\
M(s,s,\xi_0)=I.
\end{array}
\right.
\end{equation*}
Due to \eqref{Middle_F_Estimate}, then the estimate 
\begin{equation*}
|M(t,s,\xi_0)|\leq (C_3 +|\xi_0|^2 t)e^{-C (t-s)}
\end{equation*}
holds for all $0\leq s\leq t.$ If $\xi$ is near $\xi_0$, then it holds that 
\begin{equation*}
\left|\Phi(i\xi)-\Phi(i\xi_0)\right|\leq \varepsilon, 
\end{equation*}
for some $\varepsilon>0$. As in \cite[Proposition 3.3]{Reiss}, by Duhamel's principle and Gronwall's inequality, we can obtain from \eqref{Solution_Equation} the estimate
\begin{equation*}
|\hat{U}(\xi,t)|\leq (C_3 +|\xi|^2 t)e^{-(C-\varepsilon)t}. 
\end{equation*}
Hence, for $\varepsilon$ small enough, we deduce \eqref{Middle_F_Estimate}.   
\end{proof}
\subsection{Proof of Theorem \ref{Theorem_One_Damping}}\label{Theorem_One_Damping_1}
In this section, we show the decay estimate of the solution. To achieve this, we use the pointwise  estimates obtained above.

Applying the Plancherel theorem and using the first estimate in  (\ref%
{estimate_fourier_semigroup1_gamma_0_gamma_0}), we obtain%
\begin{eqnarray}
\left\Vert \partial _{x}^{j}U\left( t\right) \right\Vert _{L^{2}}^{2}
&=&\int_{\mathbb{R} }\vert \xi \vert ^{2j}\vert \hat{U}%
\left( \xi ,t\right) \vert ^{2}d\xi  \notag \\
&= &\int_{\Upsilon_L}\vert \xi \vert ^{2j}\vert \hat{U}%
\left( \xi ,t\right) \vert ^{2}d\xi +\int_{\Upsilon_M}\vert \xi \vert ^{2j}\vert \hat{U}%
\left( \xi ,t\right) \vert ^{2}d\xi+\int_{\Upsilon_H}\vert \xi \vert ^{2j}\vert \hat{U}%
\left( \xi ,t\right) \vert ^{2}d\xi
  \label{deron_U_equality_Y_L}
\end{eqnarray}
Using Proposition \ref{Proposition_Low}, we have in the low frequency region
\begin{eqnarray}\label{Low_k_1}
\int_{\Upsilon_L}\vert \xi \vert ^{2j}\vert \hat{U}%
\left( \xi ,t\right) \vert ^{2}d\xi&\leq &C\int_{\Upsilon_L}\vert \xi \vert ^{2j}e^{-c |\xi|^2 t}\vert \hat{U}( \xi ,0)\vert^2d\xi\notag\\
&\leq &C\Vert \hat{U}_{0}\Vert _{L^\infty }^{2}\int_{\left\vert
\xi \right\vert \leq 1}\left\vert \xi \right\vert ^{2j}e^{-c|\xi|
^{2}t}d\xi\notag\\
& \leq &C\left( 1+t\right) ^{-\frac{1}{2}\left( 1+2j\right)
}\left\Vert U_{0}\right\Vert _{L^{1}}^{2}.
\end{eqnarray}

In the high frequency region we distinguish two case:

First, we assume that $a=k=1$, then, we get, by using \eqref{High_F_Estimate_1}
\begin{eqnarray}\label{High_Estimate_k}
\int_{\Upsilon_H}\vert \xi \vert ^{2j}\vert \hat{U}%
\left( \xi ,t\right) \vert ^{2}d\xi&\leq &C\int_{\Upsilon_H}\vert \xi \vert ^{2(j+2)}e^{-c |\xi|^{-2} t}\vert \hat{U}( \xi ,0)\vert^2d\xi\notag\\
&\leq &C\sup_{\left\vert \xi \right\vert \geq 1}\left( \left\vert \xi
\right\vert ^{-2\ell}e^{-c\xi ^{-2}t}\right) \int_{\left\vert \xi \right\vert
\geq 1}\left\vert \xi \right\vert ^{2\left( j+\ell+2\right) }\vert \hat{U}%
\left( \xi ,0\right) \vert ^{2}d\xi  \notag\notag \\
&\leq &C\left( 1+t\right) ^{-\ell}\left\Vert \partial
_{x}^{j+\ell+1}U_{0}\right\Vert _{L^2}^{2}.
\end{eqnarray}
Second, for $a\neq 1$, we have, buy using \eqref{High_F_Estimate_2}, 
\begin{eqnarray}\label{High_Estimate_k_2}
\int_{\Upsilon_H}\vert \xi \vert ^{2j}\vert \hat{U}%
\left( \xi ,t\right) \vert ^{2}d\xi&\leq &C\int_{\Upsilon_H}\vert \xi \vert ^{2(j+6)}e^{-c |\xi|^{-2} t}\vert \hat{U}( \xi ,0)\vert^2d\xi\notag\\
&\leq &C\sup_{\left\vert \xi \right\vert \geq 1}\left( \left\vert \xi
\right\vert ^{-2\ell}e^{-c\xi ^{-2}t}\right) \int_{\left\vert \xi \right\vert
\geq 1}\left\vert \xi \right\vert ^{2\left( j+\ell+6\right) }\vert \hat{U}%
\left( \xi ,0\right) \vert ^{2}d\xi  \notag\notag \\
&\leq &C\left( 1+t\right) ^{-\ell}\left\Vert \partial
_{x}^{j+\ell+3}U_{0}\right\Vert _{L^2}^{2}.
\end{eqnarray}
For the middle frequency region, we have by using Proposition \ref{Proposition_Middle}, 
\begin{eqnarray}\label{Middle_k_Estimate}
\int_{\Upsilon_M}\vert \xi \vert ^{2j}\vert \hat{U}%
\left( \xi ,t\right) \vert ^{2}d\xi&\leq &C\int_{\Upsilon_M}\vert \xi \vert ^{2j}(1+|\xi|^{2m}t^m) e^{-C t}\vert \hat{U}( \xi ,0)\vert^2d\xi\notag\\
&\leq& Ce^{-ct}\Vert \partial_x^j U_0\Vert_{L^2},
\end{eqnarray}
where we have used the estimate 
\begin{equation*}
\int_{0}^{\infty}\left\vert \xi \right\vert ^{\sigma }e^{-c\xi ^{2}t}d\xi \leq
C\left( 1+t\right) ^{-\left( \sigma +1\right) /2}.
\end{equation*}%

Collecting \eqref{Low_k_1}, \eqref{High_Estimate_k}, \eqref{High_Estimate_k_2} and \eqref{Middle_k_Estimate}, then the estimates in Theorem \ref{Theorem_One_Damping} are fulfilled.

\end{document}